\documentclass[11pt,a4paper]{amsart}
\usepackage[margin=1.25in]{geometry}
\usepackage{amsmath, amsxtra, amsthm, amsfonts, amssymb, mathtools}
\usepackage{mathrsfs}
\usepackage{graphicx}
\usepackage{url}
\usepackage[dvipsnames]{xcolor}
\usepackage{bbm}
\usepackage{bm}
\usepackage{tikz-cd}
\usepackage{tikz}
\usetikzlibrary{backgrounds}
\usepackage[cal=boondoxo,scr=euler]{mathalfa}
\usepackage{float}
\usetikzlibrary{decorations.pathreplacing, positioning}
\usetikzlibrary{decorations.markings}
\usetikzlibrary{shapes,positioning,intersections,quotes}
\usetikzlibrary{shapes,fit}
\usetikzlibrary{calc}
\usetikzlibrary{backgrounds, fit, calc, positioning, matrix, arrows.meta}
\usepackage{comment}
\usepackage[shortlabels]{enumitem} 
\usepackage{enumitem}
\usepackage{relsize}
\usepackage{stmaryrd}
\usepackage[all,cmtip]{xy}	
\xyoption{arrow}
\usepackage[T1]{fontenc}
\usepackage{enumitem}

\usepackage[backref=page,linktocpage]{hyperref} 
\hypersetup{
    colorlinks,
    allcolors=teal
}

\usepackage[capitalize]{cleveref}

\definecolor{purplecolor}{RGB}{180, 180, 230}      
\definecolor{purpleedge}{RGB}{140, 130, 200}       
\definecolor{purpledot}{RGB}{80, 60, 140}          
\definecolor{orangecolor}{RGB}{255, 180, 100}      
\definecolor{orangeedge}{RGB}{255, 180, 100}       
\definecolor{orangedot}{RGB}{230, 130, 50}         
\definecolor{tealcolor}{RGB}{150, 210, 200}        
\definecolor{tealedge}{RGB}{100, 180, 170}         
\definecolor{tealdot}{RGB}{40, 120, 110}           

\newtheoremstyle{results}
{8pt}
{8pt}
{\itshape}
{}
{\bfseries}
{}
{.5em}
{}
\theoremstyle{results}
\newtheorem{thm}{Theorem}[section]
\newtheorem{proposition}[thm]{Proposition}
\newtheorem{corollary}[thm]{Corollary}
\newtheorem{lemma}[thm]{Lemma}
\newtheorem{claim}{Claim}
\newtheorem{maintheorem}[thm]{Theorem}

\newtheoremstyle{definitions}
{7pt}
{7pt}
{}
{}
{\bfseries}
{}
{.5em}
{}

\theoremstyle{definitions}
\newtheorem{definition}[thm]{Definition}
\newtheorem{conjecture}[thm]{Conjecture}

\newtheorem{example}[thm]{Example}

\newtheorem{remark}[thm]{Remark}

\newtheorem{problem}[thm]{Problem}

\newtheorem{question}[thm]{Question}

\crefname{thm}{theorem}{theorems}
\Crefname{thm}{Theorem}{Theorems}
\crefname{lemma}{lemma}{lemmas}
\Crefname{lemma}{Lemma}{Lemmas}
\crefname{proposition}{proposition}{propositions}
\Crefname{proposition}{Proposition}{Propositions}
\crefname{corollary}{corollary}{corollaries}
\Crefname{corollary}{Corollary}{Corollaries}
\crefname{definition}{definition}{definitions}
\Crefname{definition}{Definition}{Definitions}
\crefname{conjecture}{conjecture}{conjectures}
\Crefname{conjecture}{Conjecture}{Conjectures}
\crefname{example}{example}{examples}
\Crefname{example}{Example}{Examples}
\crefname{remark}{remark}{remarks}
\Crefname{remark}{Remark}{Remarks}

\def\Z{\mathbb{Z}}

\def\A{\mathbb{A}}
\def\N{\mathbb{N}}
\def\Z{\mathbb{Z}}

\def\calM{\mathcal{M}}
\def\calN{\mathcal{N}}

\def\calS{\mathcal{S}}

\def\N{\mathbf{N}}
\newcommand{\w}{\mathbf{w}}

\newcommand{\abs}[1]{\left\lvert#1\right\rvert}

\newcommand{\B}{\textrm{B}}

\newcommand{\G}{\textrm{G}}

\newcommand{\aug}{\operatorname{aug}}

\newcommand{\Cone}{\operatorname{Cone}}

\newcommand{\Tr}{\operatorname{Tr}}

\newcommand{\trunc}{\Tr}

\newcommand{\M}{\mathrm{M}}
\renewcommand{\emptyset}{\varnothing}

\newcommand{\floor}[1]{\lfloor #1 \rfloor}

\renewcommand{\L}{\mathcal{L}}
\newcommand{\zero}{\hat{0}}

\renewcommand{\S}{\calS}

\newcommand{\des}{\mathrm{des}}

\newcommand{\un}{\hat{1}}
\newcommand{\At}{\mathrm{At}}

\newcommand{\Edge}{\mathcal{E}}
\renewcommand{\G}{\mathcal{G}}
\newcommand{\lexop}{\vartriangleleft_{\mathrm{lex}}^{\mathrm{op}}}
\newcommand{\lk}{\mathrm{lk}}

\renewcommand{\L}{\mathcal{L}}

\renewcommand{\N}{\mathcal{N}}

\renewcommand{\hat}{\widehat}
\newcommand{\Comp}{\mathrm{Comp}}

\renewcommand{\tilde}{\widetilde}

\renewcommand{\A}{\mathcal{A}}

\newcommand{\bigast}{\mathop{\scalebox{2.1}{$\ast$}}\displaylimits}

\newcommand{\lomeg}{\lambda_{\omega}}
\newcommand{\Irr}{\mathcal{I}}
\newcommand{\StP}{Q}

\title{Structural properties of nested set complexes}

\author{Basile Coron}
\address{B. Coron, Centre de Mathématiques Laurent Schwartz, Polytechnique, Paris, France}
\email{basile.coron@polytechnique.edu}

\author{Luis Ferroni}
\address{L. Ferroni, Dipartimento di Matematica, Universit\`a di Pisa, Pisa, Italy}\email{luis.ferroni@unipi.it}

\author{Shiyue Li}
\address{S. Li, Department of Mathematics, University of Michigan, Ann Arbor, USA}\email{shiyueli@umich.edu}
\date{\today}

\begin{document}
\begin{abstract}
    We study structural and topological properties of nested set complexes of matroids with arbitrary building sets, proving that these complexes are vertex decomposable and admit convex ear decompositions. These results unify and generalize several recent and classical theorems about nested set complexes, Bergman complexes, and augmented Bergman complexes of matroids. As a first application, we show that the $h$-vector of a nested set complex is strongly flawless and, in particular, top-heavy. We then specialize to the boundary complex of the Deligne--Mumford--Knudsen moduli space $\overline{\mathcal{M}}_{0, n}$ of rational stable marked curves, which coincides with the complex of trees, establishing new structural decomposition theorems and deriving combinatorial formulas for its face enumeration polynomials.
\end{abstract}

\maketitle
\setcounter{tocdepth}{1}\tableofcontents

\section{Introduction}\label{sec:intro}

\subsection{Overview}

In \cite{de1995wonderful}, De~Concini and Procesi introduced the wonderful compactification of the complement of a hyperplane arrangement in a projective space. This compactification is obtained by iteratively blowing up the ambient projective space along certain intersections of the hyperplanes. The boundary divisor created after the sequence of blow-ups is a simple normal crossings divisor if the set of intersections being blown-up is a \textit{building set}. The elements of the building set index the irreducible boundary divisors, and those divisors meet nontrivially if and only if their indices are \emph{nested}. Thus, the nested sets form a simplicial complex, called the nested set complex, whose faces are in bijection with the boundary strata of the wonderful compactification. This notion traces back to Fulton--MacPherson's compactification of configuration spaces \cite{fulton1994compactification}, and it also features prominently in the intersection theory of matroid and polymatroid Chow rings \cite{feichtner-yuzvinsky}: nested set complexes allow us to compute Gr\"obner bases for the defining ideals of those rings (see \cite{feichtner-yuzvinsky,pagaria-pezzoli}), and triangulate the Bergman complex \cite{feichtner2004matroid}, central in the combinatorial Hodge theory of matroids.

Passing over the inherent geometric origin of De~Concini and Procesi's construction, a more general and combinatorial version of nested set complexes was introduced by Feichtner and Kozlov in \cite{feichtner-kozlov}. Concretely, to any finite meet-semilattice $\mathcal{L}$ and any \emph{building set} $\mathcal{G}$ on $\mathcal{L}$, 
one can associate a simplicial complex $\mathcal{N}(\mathcal{L},\mathcal{G})$ called the \emph{nested set complex} of the pair $(\mathcal{L},\mathcal{G})$. The \emph{leitmotiv} of the present article is to study topological and combinatorial properties that pertain the structure of nested set complexes, mainly in the case when $\mathcal{L}$ is a geometric lattice, i.e., the lattice of flats of a matroid. As we demonstrate below, the structural properties we address in this article manifest at the level of numerical invariants of the matroid.  

A classical result by Feichtner and M\"uller \cite{feichtner-muller} establishes that for any geometric lattice $\mathcal{L}$ and any building set $\mathcal{G}$ the nested set complex $\mathcal{N}(\mathcal{L},\mathcal{G})$ is Cohen--Macaulay. The proof of this assertion relies on three facts: i) Cohen--Macaulayness is a topological property, ii) the topology of $\mathcal{N}(\mathcal{L},\mathcal{G})$ is independent of $\mathcal{G}$, and iii) choosing $\mathcal{G}$ to be the \emph{maximal} building set, the complex $\mathcal{N}(\mathcal{L},\mathcal{G})$ is the order complex of $\L$ which is shellable (due to a result of Bj\"orner \cite{bjorner-el-shellable}) and therefore Cohen--Macaulay.

\subsection{Main results}

A powerful property of simplicial complexes introduced by Provan and Billera \cite{provan-billera} is \emph{vertex decomposability}. Roughly speaking, a vertex decomposable complex admits an inductive dismantling process: one can repeatedly remove a suitably chosen vertex so that both the deletion and the link of that vertex remain vertex decomposable, and no new facets are created in the deletion. Vertex decomposability is a strong refinement of shellability: it not only guarantees the existence of shellings, but provides a canonical inductive mechanism to construct them. 

The property of being vertex decomposable, as well as the property of being shellable, is not closed under homeomorphisms. In particular, we are not able to use a topological reasoning like the one used to establish the Cohen--Macaulayness of $\N(\L,\G)$ as above. By providing a careful combinatorial analysis of the behavior of links and deletions of nested set complexes, we establish the following main result. 

\begin{maintheorem}\label{thm:vertex decomposability}
    Let $\L$ be a graded lattice and let $\G$ be any building set of $\L$. If $\L$ admits an injective admissible map, then the nested set complex $\N(\L, \G)$ is vertex decomposable. 
\end{maintheorem}

The notion of admissible map was defined by Stanley in \cite{stanley-74}. It is known that any finite upper-semimodular lattice (in particular, any geometric lattice) admits an injective admissible map (see~\cite[Proposition~2.2]{stanley-74}). Therefore, we have the following immediate corollary.

\begin{corollary}\label{coro:geometric-lattices-vd}
    Let $\L(\M)$ be the lattice of flats of a matroid $\M$, and let $\G$ be any building set on $\mathcal{L}(\M)$. The nested set complex $\N(\L(\M), \G)$ is vertex decomposable and, in particular, shellable.
\end{corollary}

As we explain below, in Example~\ref{ex:augmented-bergman}, the above corollary unifies and provides a far-reaching generalization of the fact that Bergman complexes and augmented Bergman complexes of matroids are vertex decomposable (for the first, see \cite[Proposition~11.3]{BjornerWachs1997}, and for the second see \cite{amzi-jeffs}). As a consequence of Corollary \ref{coro:geometric-lattices-vd} we are able to derive a new combinatorial formula (Proposition \ref{prop:h-poly-descents}) for the $h$-vector of any nested set complex $\N(\L, \G)$ with $\L$ a geometric lattice, in terms of counting maximal nested sets with a certain descent statistics. It is natural to inquire about the numerical shape of those $h$-vectors. 

\medskip 

Given an integer $d \geqslant 0$ and a sequence of nonnegative real numbers $(h_0, \ldots, h_d)$, we say that it is \emph{flawless} or \emph{top-heavy} if, for $0 \leqslant i \leqslant \floor{d/2}$, 
\begin{align*}
h_i \leqslant h_{d-i}. \tag{\text{flawless}}
\end{align*} We say that it is \emph{strongly flawless}, if it satisfies the inequalities
\begin{align*}
h_0 \leqslant h_1 \leqslant \cdots \leqslant h_{\lfloor d/2 \rfloor}, \quad \text{ and } \quad h_i \leqslant h_{d-i}, \text{ for $0 \leqslant i \leqslant \floor{d/2}$}. \tag{strongly flawless}
\end{align*} 
Strongly flawless sequences are ubiquitous across algebraic geometry, topology, and combinatorics. A prominent example of strongly flawless sequence is the cardinality sequence of a Bruhat interval, or equivalently, the dimension sequence of the $\ell$-adic cohomology of a Schubert variety, both associated to a crystallographic Coxeter group~\cite{bjorner2009shape}. Another source of examples comes from the $h$-vectors of several simplicial complexes associated to a matroid: the independence complex~\cite{chari}, the broken circuit complex~\cite{ardila-denham-huh}, the Bergman complex~\cite{nyman-swartz}, and the augmented Bergman complex~\cite{athanasiadis-ferroni} are shown to have top-heavy $h$-vectors. Another relevant example linked to matroids originates in the singular Hodge theory of matroids: the sequence of Whitney numbers of the second kind of matroids form a strongly flawless sequence, due to the existence of a Lefschetz module structure on the intersection cohomology of matroid Schubert varieties \cite{braden2020singular}.

It is well-known that vertex decomposability and shellability are not useful to deduce inequalities beyond the Macaulay inequalities that any Cohen--Macaulay complex satisfies (see Theorem~\ref{thm:equiv-h-vector} below). In particular we are required to employ a different approach to deduce more useful information on the shape of the $h$-vectors of nested set complexes. To this end, we prove the following result.

\begin{maintheorem}\label{thm:cvx-ear-dec}
    Let $\L(\M)$ be the lattice of flats of a matroid $\M$ and let $\G$ be an arbitrary building set. The nested set complex $\N(\L(\M),\G)$ admits a convex ear decomposition. 
\end{maintheorem}

Convex ear decompositions were introduced by Chari in \cite{chari}. Roughly speaking, a simplicial complex $\Delta$ admits a convex ear decomposition if it can be built inductively by gluing together a sequence of simplicial balls that come from convex geometry. One starts from the boundary complex of a simplicial polytope, and then attaches subsequent ``ears'' one at a time along their entire boundary, so that each new piece fills in a ball whose boundary is already present in the previously constructed complex. In this way, the complex is assembled from polytope-like building blocks, attached in a maximally controlled fashion.

Similar to vertex decomposability, admitting a convex ear decomposition is not a topological property. On the other hand, as hinted from the above discussion, one reason convex ear decompositions are particularly interesting is because they imply, among other inequalities, the strong flawlessness of the $h$-vector of the complex.

\begin{corollary}\label{cor:inequalities}
     Let $\L(\M)$ be the lattice of flats of a matroid $\M$ of rank $r+1$ and let $\G$ be an arbitrary building set. The $h$-vector of $\N(\L(\M),\G)$, denoted $(h_0,h_1,\ldots,h_{r})$, is strongly flawless, i.e.,
    \begin{equation}\label{eq:top-heavy} 
    h_0 \leqslant h_1 \leqslant \cdots \leqslant h_{\lfloor r/2\rfloor}\qquad \text{and}\qquad 
     h_i \leqslant h_{r-i}\qquad \text{for each $1\leqslant i \leqslant \lfloor r/2\rfloor$.}
    \end{equation}
    Moreover,
    the $g$-vector, $(h_0, h_1-h_0, h_2-h_1, \ldots, h_{\lfloor r/2 \rfloor} - h_{\lfloor r/2 \rfloor - 1})$,
    is an $M$-vector, and the complementary vector,
    $(h_r-h_0, h_{r-1}-h_{1}, \ldots, h_{\lfloor (r+1)/2\rfloor} -h_{\lceil (r-1)/2\rceil})$,
    is a sum of $M$-vectors.
\end{corollary}

The inequalities in \eqref{eq:top-heavy} follow from results by Chari \cite[Corollary~2]{chari}. The fact that the $g$-vector of a simplicial complex with a convex ear decomposition is an $M$-vector was proved by Swartz in \cite[Corollary~3.10]{swartz}. The fact that the complementary vector of a simplicial complex with a convex ear decomposition is a sum of $M$-vectors follows from results by Swartz \cite[Section~4]{swartz} or, in a more general context, by work of Larson and Stapledon \cite{larson-stapledon}.

The above results generalizes and extends results by Nyman and Swartz \cite{nyman-swartz} and by Athanasiadis and Ferroni \cite{athanasiadis-ferroni}. Despite the fact that the inequalities \eqref{eq:top-heavy} at first glance suggest a unimodality phenomenon, it is not true in general that the $h$-vector of the nested set complex of a geometric lattice with an arbitrary building set is unimodal (see Section~\ref{sec:unimodality-real-roots}).

\medskip

We explore the consequences of our results in a relevant special case: when the poset $\L$ is the partition lattice $\Pi_n$, i.e., the set of all partitions of $[n]=\{1,2,\ldots,n\}$ ordered by coarsening, and the building set is the \emph{minimal building set} $\G = \G_{\min}$. The lattice $\Pi_n$ is the intersection lattice of the $n$-th braid arrangement. The wonderful compactification of the complement of the $n$-th braid arrangement with respect to the building set $\G_{\min}$ can be identified with the Deligne--Mumford--Knudsen compactification $\overline{\calM}_{0,n+1}$, and so $\N(\Pi_n,\G_{\min})$ is isomorphic to the boundary complex of $\overline{\calM}_{0,n+1}$. That complex is featured extensively in the literature, and has been studied from many different perspectives: tropical \cite{mikhalkin2007moduli, allcock2022tropical}, representation-theoretic \cite{robinson2004partition}, and geometric group-theoretic \cite{vogtmann1990local}. Feichtner proved in \cite{feichtner-braid} that $\N(\Pi_n, \G_{\min})$ is isomorphic to the \emph{complex of trees} $T_n$. Classical results by Trappmann and Ziegler \cite{trappmann-ziegler} (and, independently, Wachs) show that $T_n$ (and thus $\N(\Pi_n,\G_{\min})$) is shellable, and homotopy equivalent to a wedge of $(n-1)!$ spheres of dimension $n-3$ \cite{vogtmann1990local, robinson2004partition}. Our main results imply something stronger: the complex of trees is vertex decomposable and admits a convex ear decomposition (see Corollary~\ref{coro:trees-vd}).  

Using our new combinatorial formula for the $h$-vector of nested set complexes given by Proposition \ref{prop:h-poly-descents}, we state explicitly a combinatorial description of the $h$-vector of $\N(\Pi_n,\G_{\min})$ is terms of counting the so-called Stirling permutations introduced by Gessel and Stanley~\cite{gessel-stanley}, and we deduce new inequalities between the entries of the Second Eulerian polynomials (see Corollary~\ref{coro:stirling-top-heavy}). 

\subsection*{Related work}

In the final stages of this project, we learned that Backman, Dorpalen-Barry, Nathanson, Partida, and Prime \cite{backman-dorpalen-nathanson-partida-prime} had independently proved the shellability of nested set complexes of geometric lattices and arbitrary building sets via the tropical geometry of normal complexes. Their methods and ideas complement those developed in the present manuscript. For this reason, we decided to post our results on the same day.

\subsection*{Acknowledgments}
We thank Spencer Backman, Galen Dorpalen-Barry, Anastasia Nathanson, Ethan Partida, and Noah Prime for kindly discussing their results on shellability in \cite{backman-dorpalen-nathanson-partida-prime}. We are also grateful to Christos Athanasiadis for valuable discussions on convex ear decompositions, and to Matt Larson for comments on doubly Cohen--Macaulay complexes. 

\section{Background}\label{sec:prelim}

We begin by recalling basic definitions and by reviewing the relationships among several structural, topological, and numerical properties of simplicial complexes. To give a sense of the landscape, we include \Cref{fig:properties} summarizing the implications among the properties relevant to this paper.

A \emph{simplicial complex} $\Delta$ on a finite vertex set $V$ is a collection of subsets of $V$, called \emph{faces}, which is closed under inclusion: if $F \in \Delta$ and $G \subseteq F$, then $G \in \Delta$. We also require that $\Delta$ contains every singleton in $V$. The \emph{dimension} of a face $F$ is $\dim F = |F| - 1$, and the dimension of $\Delta$ is the maximum dimension of its faces. A simplicial complex is \emph{pure} if all its maximal faces, called \emph{facets}, have the same dimension.  A \emph{subcomplex} $\Delta'$ of a simplicial complex $\Delta$ is a subset of $\Delta$ which is a simplicial complex.  For more background on simplicial complexes, we refer to \cite{stanley-96}. 

We set notations for several operations on simplicial complexes. For a simplicial complex $\Delta$ of dimension $d$ and a nonempty face $F \in \Delta$, the \emph{deletion} of $F$ from $\Delta$ is 
\[
\Delta \setminus F \coloneqq \{G \in \Delta \mid F \not\subseteq G\}. 
\]
The \emph{link} of a nonempty face $F$ in $\Delta$ is the subcomplex $\lk_{\Delta}(F)$ defined by 
\[
\lk_{\Delta}(F) \coloneq \{G \in \Delta \mid F \cap G = \varnothing, F \cup G \in \Delta\}. 
\] 
For simplicial complexes $\Delta$ and $\Delta'$ on disjoint vertex sets $V$ and $V'$, the \emph{join} of $\Delta$ and $\Delta'$ is the simplicial complex on the vertex set $V \sqcup V'$ defined by
\[
\Delta \ast \Delta' \coloneqq \{F \cup F' \mid F \in \Delta,\, F' \in \Delta'\}. 
\]
For faces $F_1, \ldots, F_k$, we write $\langle F_1, \ldots, F_k \rangle$ for the smallest subcomplex of $\Delta$ containing all faces $F_1, \ldots, F_k$.

\begin{figure}[ht]
\begin{tikzpicture}[
    box/.style={draw, rounded corners, inner sep=4pt, minimum height=1.8em},
    impl/.style={double, double distance=2pt, -Implies, shorten >=3pt, shorten <=3pt},
    every matrix/.style={column sep=0.5em, row sep=1.2cm}
]
\matrix (m) [matrix of nodes, nodes={box}] {
    |[name=vd]| vertex decomposability & |[name=ps]| PS ear decomposition &  \\
    |[name=sh]| shellability & & |[name=ced]| convex ear decomposition \\
    |[name=cm]| Cohen-Macaulay & |[name=dcm]| doubly Cohen-Macaulay &  \\
    |[name=nh]| nonnegative $h$ & |[name=sff]| strongly flawless $f$ & |[name=sfh]| strongly flawless $h$ \\
};

\begin{scope}[on background layer]
    \path (vd.west |- vd.north) ++(-10pt, 5pt) coordinate (TL);
    \path (ced.east |- nh.south) ++(10pt, -5pt) coordinate (BR);
    
    \fill[blue!10, rounded corners=8pt] 
        ([yshift=15pt]TL) rectangle ([yshift=-15pt]BR |- sh.south);
    
    \fill[orange!15, rounded corners=8pt] 
        ([yshift=15pt]TL |- cm.north) rectangle ([yshift=-15pt]BR |- cm.south);
    
    \fill[teal!15, rounded corners=8pt] 
        ([yshift=15pt]TL |- nh.north) rectangle ([yshift=-15pt]BR);
\end{scope}

\node[anchor=north east, font=\small, color=blue!50] 
    at ([xshift=-5pt, yshift=15pt]BR |- vd.north) {\scshape structural};
\node[anchor=north east, font=\small, color=orange!70] 
    at ([xshift=-5pt, yshift=15pt]BR |- cm.north) {\scshape topological};
\node[anchor=north east, font=\small, color=teal!70] 
    at ([xshift=-5pt, yshift=15pt]BR |- nh.north) {\scshape numerical};

\draw[impl] (vd) -- (sh) node[midway, left] {{\footnotesize \cite{provan-billera}}};
\draw[impl] (ps) -- (ced);
\draw[impl] (ps) -- (sh) node[midway, right=5mm] {\footnotesize \cite{chari}};
\draw[impl] (sh) -- (cm) node[pos=0.3, left] {\footnotesize\cite{stanley-96}};
\draw[impl] (ced) -- (sfh) node[midway, left] {\footnotesize\cite{chari}};
\draw[impl] (ced) -- (dcm) node[midway, left=10mm, above] {\footnotesize\cite{swartz}};
\draw[impl] (dcm) -- (cm);
\draw[impl] (cm) -- (nh) node[pos=0.3, left] {\footnotesize\cite{stanley-96}};
\draw[impl] (cm) -- (sff);
\end{tikzpicture}
\caption{Relationships among the structural, topological, and numerical properties of simplicial complexes relevant to the present paper.}
\label{fig:properties}
\end{figure}
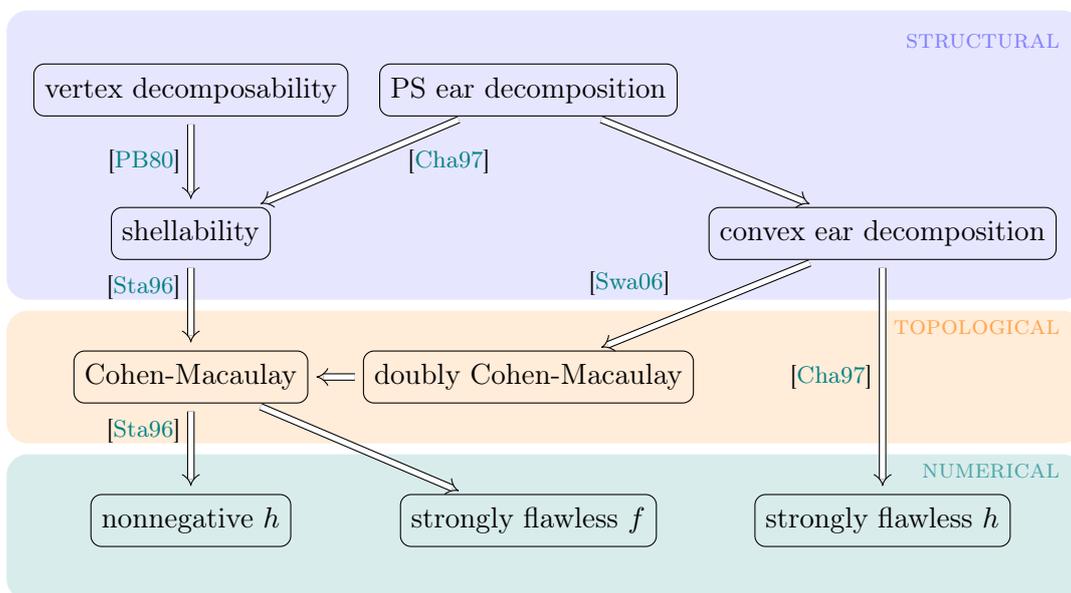

\subsection{Cohen-Macaulayness and shellability} 
A simplicial complex $\Delta$ is said to be \emph{Cohen--Macaulay} over a field $K$ (or over the ring $\mathbb{Z}$)\footnote{If $\Delta$ is Cohen--Macaulay over $\mathbb{Z}$, then $\Delta$ is Cohen--Macaulay over any field $K$. If $\Delta$ is Cohen--Macaulay over a field $K$, then $\Delta$ is Cohen--Macaulay over $\mathbb{Q}$.} if, for every $F \in \Delta$, the reduced homology group $\widetilde{H}_i(\lk_{\Delta}(F); K)$ is trivial for every $1 \leqslant i < \dim \lk_{\Delta}(F)$. Simplicial complexes that are Cohen--Macaulay over a field $K$ can also be characterized as those complexes whose Stanley--Reisner ring is a Cohen--Macaulay ring over $K$. For background on combinatorics and commutative algebra we refer to \cite{stanley-96}. By a result of Munkres (see \cite[Proposition~4.3]{stanley-96}) Cohen--Macaulayness only depends on the topology of the geometric realization of $\Delta$, and therefore, is a topological property. Shellability, on the other hand, is a structural property that is often used to prove Cohen--Macaulayness by combinatorial methods. 

\begin{definition}
    A $(d-1)$-dimensional simplicial complex $\Delta$ is \emph{shellable} if $\Delta$ is pure and there exists an ordering $F_1, F_2, \ldots, F_k$ of its facets if for $2 \leqslant i \leqslant k$, the subcomplex 
    \[
    \langle F_1, \ldots, F_{i-1}\rangle \cap \langle F_i \rangle \subset \Delta
    \] is pure of dimension $d-2$. 
\end{definition}
Despite the strong topological consequences of being shellable, it is not true that shellability is a topological property. 

Many simplicial complexes arising in algebraic and geometric aspects of matroid theory are known to be shellable and thus also Cohen--Macaulay. This includes classical examples such as independence complexes, broken-circuit complexes, Bergman complexes, and augmented Bergman complexes (for the first three, we refer to \cite{bjorner}, whereas for the last we refer to \cite{bullock-et-al,amzi-jeffs}).  

\begin{definition}
    A Cohen--Macaulay simplicial complex $\Delta$ is said to be \emph{doubly Cohen--Macaulay} if, for every vertex $v\in \Delta$, the complex $\Delta \setminus \{v\}$ is Cohen--Macaulay of the same dimension as $\Delta$.
\end{definition}

A one-dimensional simplicial complex $\Delta$ is doubly Cohen--Macaulay if and only if it is a biconnected graph. In particular, double Cohen--Macaulayness is often seen as a higher dimensional analogue of biconnectedness. By a result of Walker (see \cite[Proposition~3.7]{stanley-96}), double Cohen--Macaulayness is a topological property. Doubly Cohen-Macaulayness neither implies nor is implied by shellability. 

\subsection{Vertex decomposability}\label{subsec:VD}
In this section, we collect definitions and properties related to the notion of vertex decomposition, which was introduced in a classical paper by Provan and Billera \cite{provan-billera}.

\begin{definition}[{\cite[Definition 2.1]{provan-billera}}]
    A simplicial complex $\Delta$ is \emph{$k$-decomposable} if $\Delta$ is pure, and either $\Delta$ is a simplex, or inductively there exists a face $F$ of $\Delta$ with $\dim F \leqslant k$ with the following properties: 
    \begin{enumerate}[\normalfont (i),leftmargin=20pt]
        \item (Shedding) Every facet of $\Delta \setminus F$ is also a facet of $\Delta$ 
        \item (Recursive decomposability) The link $\lk_{\Delta}(F)$ is ($d-\abs{F}$)-dimensional and $k$-decomposable, and the deletion $\Delta \setminus F$ is $d$-dimensional and $k$-decomposable. 
    \end{enumerate}
\end{definition}
If a face $F \in \Delta$ satisfies the shedding condition, it is called a \emph{shedding} face. By definition, for $0 \leqslant k < \dim \Delta$, $k$-decomposability implies $(k + 1)$-decomposability. For $\dim \Delta \leqslant k$, the notions of $k$-decomposability and $(k+1)$-decomposability are equivalent. For $k = \dim \Delta$, the property of being $k$-decomposable is known to be equivalent to shellability (see \cite[Theorem~2.8]{provan-billera}). In particular being $k$-decomposable for any $k \geqslant 0$ implies shellability.

The case $k = 0$, called \emph{vertex decomposability}, is the strongest property in this hierarchy and a central focus of this paper. We specialize the definition to this case.
\begin{definition}
    A simplicial complex $\Delta$ is \emph{vertex decomposable}, if $\Delta$ is a simplex, or inductively, $\Delta$ is pure and there exists a vertex $v \in \Delta$ such that the following two properties hold: 
    \begin{enumerate}[\normalfont (i)]
        \item (Shedding) Every facet of $\Delta \setminus \{v\}$ is also a facet of $\Delta$ (in this case $v$ is said to be a \emph{shedding vertex}).
        \item (Recursive decomposability) The deletion $\Delta \setminus \{v\}$ and the link $\lk_{\Delta}(v)$ are both vertex decomposable.
    \end{enumerate}
\end{definition}

Several operations on simplicial complexes preserve vertex decomposability (we refer to \cite[Section~2]{provan-billera} for the proofs and further details). If $\Delta$ vertex decomposable, then so is $\lk_{\Delta}(F)$ for each $F\in\Delta$. The join of two simplicial complexes is vertex decomposable if and only if each of the two complexes is vertex decomposable (in particular, coning over or deconing vertex decomposable complexes yields vertex decomposable complexes). Stellar subdivisions preserve vertex decomposability.

Several classes of simplicial complexes related to matroids are known to be vertex decomposable, for example the independence complex~\cite{provan-billera}, the broken circuit complex~\cite{provan-billera}, the Bergman complex (equivalently, the order complex of a geometric lattice)~\cite{BjornerWachs1997}, and the augmented Bergman complex~\cite{amzi-jeffs}.

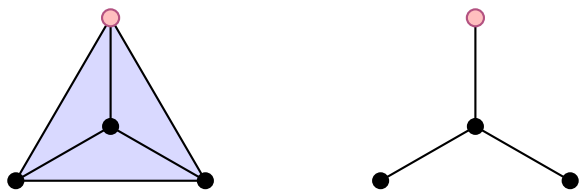
\begin{figure}[ht]
\begin{tikzpicture}[scale=0.8]
    \begin{scope}[shift={(-3,0)}]
        \fill[blue!15] (90:1.8) -- (210:1.8) -- (0,0) -- cycle;
        \fill[blue!15] (90:1.8) -- (330:1.8) -- (0,0) -- cycle;
        \fill[blue!15] (210:1.8) -- (330:1.8) -- (0,0) -- cycle;
        
        \draw[thick] (90:1.8) -- (210:1.8);
        \draw[thick] (90:1.8) -- (330:1.8);
        \draw[thick] (210:1.8) -- (330:1.8);
        
        \draw[thick] (0,0) -- (90:1.8);
        \draw[thick] (0,0) -- (210:1.8);
        \draw[thick] (0,0) -- (330:1.8);
        
        \fill[pink] (90:1.8) circle (4pt);  
        \draw[thick, magenta!70!black] (90:1.8) circle (4pt);
        \fill[black] (210:1.8) circle (4pt);
        \fill[black] (330:1.8) circle (4pt);
        
        \fill[black] (0,0) circle (4pt);
    \end{scope}
    
    \begin{scope}[shift={(3,0)}]
        \draw[thick] (0,0) -- (90:1.8);
        \draw[thick] (0,0) -- (210:1.8);
        \draw[thick] (0,0) -- (330:1.8);
        
        \fill[pink] (90:1.8) circle (4pt);  
        \draw[thick, magenta!70!black] (90:1.8) circle (4pt);
        \fill[black] (210:1.8) circle (4pt);
        \fill[black] (330:1.8) circle (4pt);
        
        \fill[black] (0,0) circle (4pt);
    \end{scope}
\end{tikzpicture}
\centering
\caption{Vertex decomposable simplicial complexes with shedding vertices.}
\label{fig:VD}
\end{figure}

\subsection{Convex ear decomposition} Convex ear decompositions were introduced by Chari \cite{chari} as a tool to study face enumerations of independence complexes of matroids. It is a higher dimensional analogue of ear decompositions of graphs. 
\begin{definition}
    For a simplicial complex $\Delta$ of dimension $d-1$, a \emph{convex ear decomposition} of $\Delta$ is a sequence $(\Delta_1, \ldots, \Delta_m)$ of subcomplexes of $\Delta$ such that 
    \begin{enumerate}[\normalfont (i)]
        \item $\Delta = \bigcup_{i=1}^{m} \Delta_{m}.$
        \item The first subcomplex $\Delta_1$ is isomorphic to the boundary complex of a $d$-dimensional simplicial polytope. 
        \item For $2 \leqslant k \leqslant m$, $\Delta_k$ is a $(d-1)$-dimensional simplicial ball which is isomorphic to a proper subcomplex of the boundary complex of a $d$-dimensional simplicial polytope, and 
        \[
        \partial \Delta_k = \Delta_k \cap \left( \bigcup_{i=1}^{k-1} \Delta_i \right). 
        \]
    \end{enumerate}
    If such a decomposition exists, each $\Delta_{i}$ for $i \geqslant 2$ is called an \emph{ear}. 
\end{definition}

\begin{figure}[ht]
    \centering
    \includegraphics[width=\textwidth]{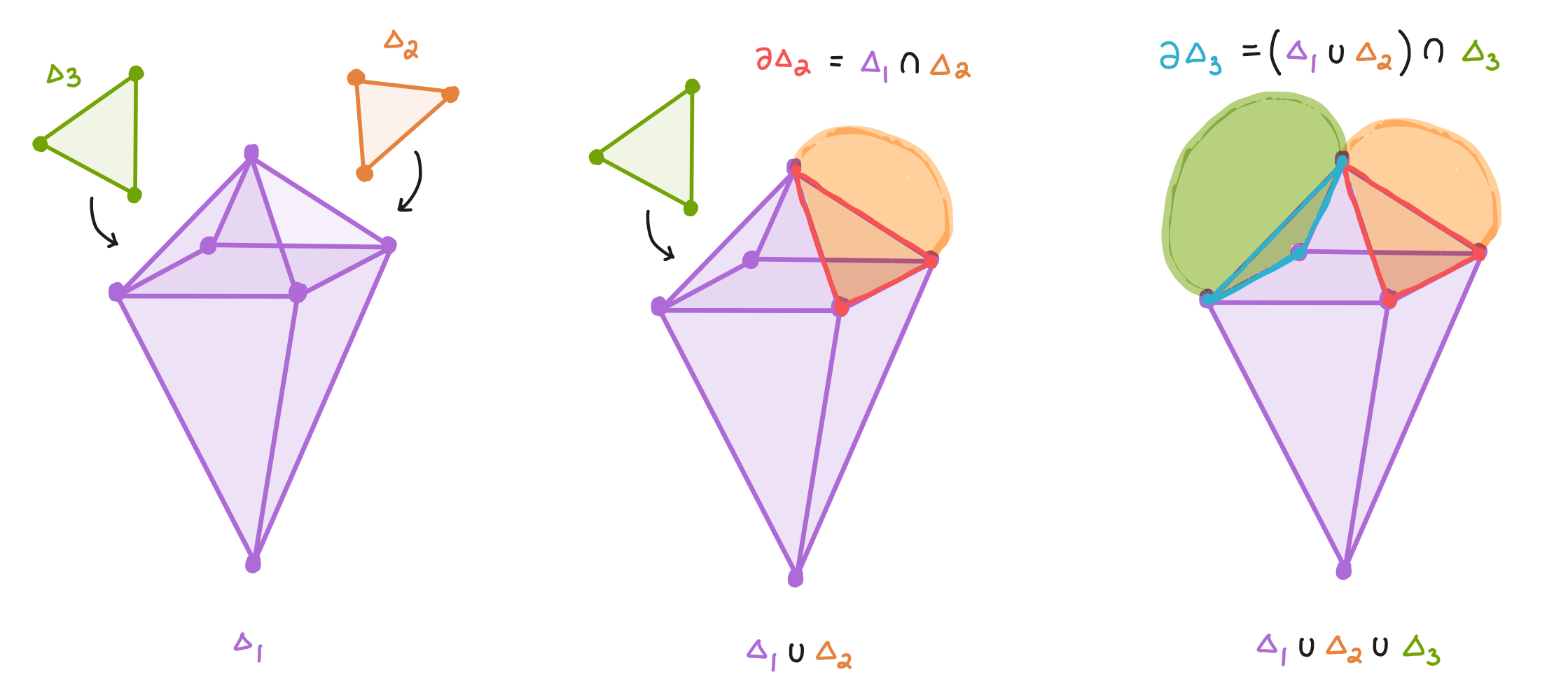}
    \caption{A convex ear decomposition of an ``ice cream cone,'' consisting of $\Delta_1$ (the boundary of the octahedron), and the ``ears'' $\Delta_2$ and $\Delta_3$, which are $2$-dimensional simplicial balls.}
\label{fig:ced}
\end{figure} 

Swartz \cite[Theorem~4.1]{swartz} proved that every simplicial complex $\Delta$ with a convex ear decomposition is doubly Cohen--Macaulay over $\Z$. Despite the fact that double Cohen--Macaulayness is a topological property, convex ear decompositions are easily seen to not be preserved under homeomorphisms.

Several classes of simplicial complexes are shown to be doubly Cohen--Macaulay via exhibiting a convex ear decomposition. Some examples related to matroids include the independence complex (if the matroid has no coloops) \cite{chari}, the Bergman complex \cite{nyman-swartz}, and the augmented Bergman complex \cite{athanasiadis-ferroni}.

A special case of convex ear decomposition is \emph{PS ear decomposition}, introduced by \cite{chari}. A PS\footnote{The terminology ``PS'' comes from the fact that products of simplices are involved in the definition.} ear decomposition is a convex ear decomposition in which $\Delta_1$ is the boundary of a simplicial polytope and each $\Delta_i$, for $i\geq 2$, is the join of a simplex and the boundary of a simplicial polytope. Chari proved in \cite[Proposition~5]{chari} that the existence of a PS ear decomposition guarantees shellability. Furthermore, in \cite[Theorem~2]{chari} he also showed that PS ear decompositions admit a recursive structure very similar to $k$-decomposability. 

\subsection{Face enumeration}
Let $\Delta$ be a finite simplicial complex of dimension $d-1$. For each integer $i\in \{0,1,\ldots,d\}$ let us denote by $f_i(\Delta)$ the number of simplices in $\Delta$ of cardinality $i$ (that is, of dimension $i-1$). 
The \emph{$f$-vector} of $\Delta$ is the $(d+1)$-tuple $(f_0(\Delta), f_1(\Delta), \dots, f_{d}(\Delta))$. The \emph{$h$-vector} of $\Delta$, denoted $(h_0(\Delta),\ldots,h_{d}(\Delta))$, is defined by the following equation
\[
  \sum_{i=0}^d f_{i}(\Delta)(y-1)^{d-i}
  \;=\;
  \sum_{i=0}^d h_i(\Delta) y^{d-i}.
\]
The $f$-polynomial $f(\Delta; x)$ and the $h$-polynomial $h(\Delta; y)$ are defined as 
\[
  f(\Delta; x) \coloneqq \sum_{i=0}^d f_{i} \: x^i,
  \quad \text{ and } \quad
  h(\Delta; y) \coloneqq \sum_{i=0}^d h_i \: y^i.
\]

The \emph{$g$-vector} of $\Delta$, denoted $(g_0(\Delta), \ldots, g_{\lfloor d/2\rfloor}(\Delta))$ is defined by $g_0(\Delta) = h_0(\Delta)$, and $g_i(\Delta) = h_{i}(\Delta)-h_{i-1}(\Delta)$ for $1\leqslant i \leqslant \lfloor d/2\rfloor$. The \emph{complementary vector} of $\Delta$, which we denote by $(c_0(\Delta), \ldots, c_{\lfloor d/2\rfloor}(\Delta))$ is defined by $c_i(\Delta) = h_{d-i}(\Delta)-h_i(\Delta)$ for each $0\leqslant i \leqslant \lfloor d/2\rfloor$.

A sequence $(h_0,h_1,\ldots,h_d)\in \mathbb{Z}^{d+1}$ is said to be an \emph{$M$-vector} if there exists a multicomplex $\Gamma$ such that, for every $0\leqslant i\leqslant d$, the number $h_i$ counts the faces of cardinality $i$ in $\Gamma$. A list of inequalities characterizing $M$-vectors was found by Macaulay (see~\cite[Theorem~II.2.2]{stanley-96}). 

The following is a classical result due to Stanley (see \cite[Theorem~II.3.3]{stanley-96}). 

\begin{thm}\label{thm:equiv-h-vector}
    Let $h=(h_0,\ldots,h_d)\in \mathbb{Z}^{d+1}$. The following are equivalent:
    \begin{enumerate}[\normalfont (i)]
        \item $h$ is an $M$-vector.
        \item $h$ is the $h$-vector of a Cohen--Macaulay complex.
        \item $h$ is the $h$-vector of a shellable complex.
    \end{enumerate}
\end{thm}

An elementary consequence of the above result is that the $f$-vector of a Cohen--Macaulay (or shellable) complex is strongly flawless (see, e.g., \cite[Prop.~7.2.5(i)]{bjorner}).

It is possible to add a fourth point to the above list of equivalent properties: ``$h$ is the $h$-vector of a vertex decomposable complex''. This follows from a result of Kalai: if $\Delta$ is a Cohen--Macaulay complex, applying an operation called (combinatorial) shifting to $\Delta$ yields a pure shifted complex $\Delta'$ with the same $h$-vector as $\Delta$ (see~\cite[Theorem~4.1]{kalai}); furthermore, any pure shifted complex is automatically vertex decomposable (see, e.g., \cite[Theorem~11.3]{BjornerWachs1997}). 

Thanks to the $g$-theorem for simplicial polytopes proved in \cite{stanley-g-theorem,billera-lee},  convex ear decompositions guarantee several inequalities for the $h$-vector of a complex.

\begin{thm}
    Let $\Delta$ be a $(d-1)$-dimensional complex admitting a convex ear decomposition. Then, the $h$-vector $h(\Delta) = (h_0,\ldots,h_d)$ is strongly flawless, i.e.,  
    \begin{equation}\label{eq:str-flawless} 
    h_0 \leqslant h_1 \leqslant \cdots \leqslant h_{\lfloor d/2\rfloor}\qquad \text{and}\qquad 
     h_i \leqslant h_{d-i}\qquad \text{for each $1\leqslant i \leqslant \lfloor d/2\rfloor$.}
    \end{equation}
    Moreover,
    the $g$-vector, $(h_0, h_1-h_0, h_2-h_1, \ldots, h_{\lfloor d/2 \rfloor} - h_{\lfloor d/2 \rfloor - 1})$,
    is an $M$-vector, and the complementary vector,
    $(h_d-h_0, h_{d-1}-h_{1}, \ldots, h_{\lceil d/2\rceil} -h_{\lfloor d/2\rfloor})$,
    is a sum of $M$-vectors.
\end{thm}

The first part of the last result is proved in \cite[Corollary~2]{chari}. The second part was established by Swartz in \cite{swartz}. 

\begin{remark}
    By using ideas related to the $g$-conjecture for simplicial spheres, Adiprasito, Papadakis, and Petrotou  announced a result \cite[Corollary~3.2]{adiprasito-papadakis-petrotou} that implies that the first part of the above theorem continues to hold for doubly Cohen--Macaulay complexes. A different proof is provided in a recent paper by Larson and Stapledon, which in addition implies a stronger version of the second part of the above statement for doubly Cohen--Macaulay complexes over a field of characteristic $2$ \cite[Corollary~1.7]{larson-stapledon}. Furthermore, Larson and Stapledon provide a characterization of all complementary vectors of simplicial complexes admitting a convex ear decomposition, and deduce additional inequalities for $h$-vectors of doubly Cohen--Macaulay complexes over a field of characteristic $2$ \cite[Corollary~4.5]{larson-stapledon}.
\end{remark}

\subsection{Admissible lattices}\label{subsec:admissible-lattices}

In this section, we introduce the class of posets which we will be focusing on throughout this article. We write $\leqslant$ for the partial order on a poset and $\lessdot$ for its covering relation. We refer to \cite{wachs2007poset} for background on posets. 

\begin{definition}
    A finite poset $\L$ is called a \textit{lattice} if every pair of elements $F_1, F_2 \in \L$ has a supremum, denoted $F_1\vee F_2$, and an infimum, denoted $F_1 \wedge F_2$.
\end{definition}

Since $\L$ is finite, the existence of suprema and infima for pairs implies the existence of suprema and infima for any subset $S \subset \L,$ denoted by $\bigvee S$ and $\bigwedge S$ respectively. In particular we denote $\un = \bigvee \L$ which is an upper bound of $\L$ and $\zero = \bigwedge \L$ which is a lower bound of $\L.$ An element in a lattice $\L$ is called \emph{join-irreducible} if it cannot be written as a join of strictly smaller elements. The set of join-irreducible elements of $\L$ will be denoted by $\Irr(\L)$. If we have a map $\omega: \Irr(\L) \rightarrow \Z_{>0}$, then $\omega$ induces a labeling $\lambda_{\omega}:\Edge(\L) \rightarrow \Z_{>0}$ of the set $\Edge(\L)$ of covering relations of $\L$, defined by 
\[
\lambda_{\omega}(F_1 \lessdot F_2) \coloneqq \min \, \{\omega(I) \, | \, I \in \Irr(\L),\, F_1 \vee I = F_2\}. 
\]

Recall that an edge-labeling is an \emph{R-labeling} if for any two elements $F_1 \leqslant F_2$ there exists a unique saturated chain between $F_1$ and $F_2$ with weakly increasing labels. The following definition is due to Stanley \cite{stanley-74}.

\begin{definition}
The map $\omega$ is called \textit{admissible} if the labeling $\lambda_{\omega}$ is an $R$-labeling. A graded lattice is called \textit{admissible} if it admits an admissible map. 
\end{definition}

Stanley \cite{stanley-74} proved that when $\L$ is semimodular, any injective map $\Irr(\L) \hookrightarrow \Z_{>0}$ is admissible. In particular, this result holds for geometric lattices. Björner \cite[Section 3]{bjorner-el-shellable} proved that if a map $\omega:\Irr(\L) \rightarrow \Z_{>0}$ is admissible, then $\lambda_{\omega}$ is in fact an EL-labeling. This in turn implies that the order complex $\Delta(\L\setminus \{\zero, \un\})$ is vertex decomposable \cite[Theorem 11.6]{BjornerWachs1997}.

\subsection{Building sets and nested sets}\label{subsec:prelim-building-sets}
In this subsection, we recall the basics of the main objects of this work: building sets and nested sets in lattices, originally introduced in \cite{feichtner-kozlov}. 

For any subset $\G \subseteq \L$ and any $F \in \L$, let $\G_{\leqslant F}$ be the collection of elements in $\G$ which are less than or equal to $F$. 
\begin{definition}\label[definition]{def:building-set}
    A subset $\G \subset \L\setminus \{\zero\}$ is called a \emph{building set} of $\L$ if for any $F \in \L$, the join map 
    \[
    \psi_{F} \colon \prod_{G \in \max \G_{\leqslant F}}[\zero, G] \xrightarrow{\sim} [\zero, F], \quad (H_G)_G \mapsto \bigvee_G H_G
    \] is an isomorphism of posets. The isomorphism $\psi_F$ is called the \emph{structural isomorphism} of $\G$ at $F$. If $\G$ is a building set, then the elements of $\max \G_{\leqslant F}$ are called the \emph{$\G$-factors} of $F$. A pair $(\L, \G)$ with $\G$ a building set of a lattice $\L$ is called a \emph{built lattice}. A built lattice $(\L, \G)$ such that $\L$ is graded is called a \textit{graded built lattice}. 
\end{definition}

\begin{example}
    For any lattice $\L$, the subset $\L\setminus \{\zero\}$ is a building set of $\L$, called the \emph{maximal building set} of $\L$ and denoted by $\G_{\max}.$ At the other extreme, the set $\G_{\min}$ of irreducible elements of $\L$, that is, the elements $G\in \L$ such that $[\zero, G]$ is not a product of proper subposets, is a building set of $\L$, called the \emph{minimal building set} of $\L$.
\end{example}

\begin{remark}\label{rmk:irreducible-in-building}
    If $(\L, \G)$ is a built lattice, the isomorphism given by the join map in \Cref{def:building-set} implies that every element of $\L$ is the join of its $\G$-factors. This implies that $\G$ must contain every join-irreducible element of $\L$, and in particular, must contain every atom of $\L.$
\end{remark}

The following lemma about factors in building sets will be useful later on.

\begin{lemma}\label{lem:factor-union}
    Let $(\L, \G)$ be a built lattice, and let $F$ be some element of $\L$ with factors $F_1,\ldots, F_k,$ giving the isomorphism $\psi_F \colon \prod_{1\leqslant i \leqslant k}[\zero, F_i] \simeq [\zero, F].$ For every $H \leqslant F,$ if we denote $(H_i)_{1 \leqslant i\leqslant k} \coloneqq \psi^{-1}(H)$, then the set of $\G$-factors of $H$ is the disjoint union of the sets of $\G$-factors of the $H_i$'s
    \[
    \max \G_{\leqslant H} = \bigsqcup_{1\leqslant i \leqslant k} \max \G_{\leqslant H_i}.
    \]
\end{lemma}

\begin{proof}
    Let $G$ be a $\G$-factor of $H$. Then there is a unique factor of $F$, say $F_1,$ with $G \leqslant F_1$. Since
    \[
    \psi(H_1 \vee G, H_2, \ldots, H_k) = \psi(H_1, H_2, \ldots, H_k) = H,
    \] and $\psi$ is injective, we must have $G \leqslant H_1,$ which implies that $G$ is a factor of $H_1.$ 

    Let $G$ be a $\G$-factor of $H_i$ for some $1 \leqslant i \leqslant k,$ and let $G' \in \G$ be such that $G \leqslant G' \leqslant H$. Since $G' \leqslant F$, the element $G'$ must be below some unique factor of $F$, and since $G\leqslant G'$ that unique factor must be $F_i$. By the same argument as above, $G'$ must then be below $H_i.$ Since $G$ is a factor of $H_i$, we must have $G = G'$, which implies that $G$ is a factor of $H.$
\end{proof}
\begin{definition}
    Let $(\L, \G)$ be a built lattice. A subset $\S \subset \G$ is a \emph{nested set} if for any antichain $\A $ of cardinality at least $2$, the join $\bigvee \A$ does not belong to $\G.$ The collection of nested sets of a built lattice $(\L, \G)$ forms an abstract simplicial complex with vertex set $\G$, which is denoted by $c\N(\L, \G).$ We denote the restriction to nonmaximal elements of $\G$ by $\N(\L, \G) \coloneqq c\N(\L, \G)\setminus \max \G.$ 
\end{definition}

\begin{example}\label{ex:nested-max}
    For any lattice $\L$, in the case of the maximal building set $\G_{\max} = \L\setminus \{\zero\},$ a subset $\S \subset \L\setminus \{\zero\}$ is nested if and only if it is a chain. In other words, the nested set complex $\N(\L, \G_{\max})$ is the order complex of $\L \setminus \{\zero, \un\}$, and the nested set complex $c\N(\L, \G_{\max})$ is the order complex of $\L\setminus \{\zero\}$. 
\end{example}

The notation $c\N(\L,\G)$ is motivated by the following lemma, which shows that we have a notion of restriction of a built lattice. Let $(\L, \G)$ be a built lattice. For any $F \in \L$, we define the \emph{restriction} of $\G$ and $\L$ at $F$ as 
\begin{align*}
    \G^F &\coloneqq \G \cap [\zero, F], \\
    \L^F &\coloneqq [\zero, F]. 
\end{align*} 
\begin{lemma}\label{lem:reducible}
For any built lattice $(\L, \G)$ and any $F \in \L\setminus\{\zero\}$, the subset $\G^F$ is a building set of the lattice $\L^{F}.$ If $G_1, \ldots, G_k$ are the maximal elements of $\G,$ then we have an isomorphism of nested set complexes
\begin{equation}\label{eq:nested-set-cone}
c\N(\L,\G) \simeq \Cone(\N(\L^{G_1}, \G^{G_1}))* \cdots * \Cone(\N(\L^{G_k}, \G^{G_k})).
\end{equation}
\end{lemma}
\begin{proof}
The first statement is a special case of \cite[Lemma 2.8.5]{BDF_2022}. For the second statement, first notice that for $1 \leqslant i \leqslant k$, a subset $\S \subset \G^{G_i}$ is nested if and only if $\S \cup \{G_i\}$ is nested, and so we get an isomorphism $c\N(\L^{G_i}, \G^{G_i}) \simeq \Cone(\N(\L^{G^i}, \G^{G^i})).$ To prove the isomorphism \eqref{eq:nested-set-cone}, it remains to prove that for every family of sets $(\S_i)_{1\leqslant i \leqslant n}$ with $\S_i \subset \G^{G_i}$, the subset $\bigsqcup_i \S_i \subset \G$ is nested if and only if $\S_i$ is nested for $1 \leqslant i \leqslant k$. If each $\S_i$ is nested, let $\A \subset \bigsqcup_i \S_i$ be an antichain of cardinality at least $2.$ If $\bigvee \A \in \G$, then $\bigvee \A \leqslant G_i$ for some $i$, but then we have $\A \subset \S_i$, which contradicts the fact that $\S_i$ is nested. The other implication is immediate. 
\end{proof}
The above lemma will allow us to restrict to the study of nested set complexes associated to built lattices $(\L,\G)$ such that $\un \in \G.$ Those built lattices will be called \emph{irreducible.} 

The next lemma shows that we also have a notion of contraction of a built lattice. For any $F \in \L$, we define the \emph{contraction} of $\G$ and $\L$ at $F$ as 
\begin{align*}
\G_F &\coloneqq \{ F \vee G \, | \, G \in \G  \} \setminus \{ F \}, \\
\L_F &\coloneqq [F, \un]. 
\end{align*} 

\begin{lemma}[{\cite[Lemma 2.8.5]{BDF_2022}}]
Let $(\L, \G)$ be a built lattice. For any $F \in \L$, the set $\G_F$
is a building set of the lattice $\L_F$. 
\end{lemma}

We have the following crucial lemma about the general structure of nested sets. 
\begin{lemma}\label{lem:nested-tree}
Let $\S$ be a nested set of a built lattice $(\L, \G).$ The set $\S$ is a tree in the order graph of $\L$: for every $F \in \L\setminus \{\zero\},$ the set $\S_{>F} = \{ G \in \S \, | \, G > F\}$ has a unique minimum.
\end{lemma}
\begin{proof}
    Suppose for contradiction that $\min \S_{>F}$ has at least two elements, say $G_1$ and $G_2$. Since $\S$ is nested, $G_1, G_2$ are the factors of $G_1\vee G_2$ in $\G$. Then the fact that $F$ belongs to both $[\zero, G_1]$ and $[\zero, G_2]$ contradicts the structural isomorphism of $\G$ at $F$
    \[
    \psi_F \colon [\zero, G_1]\times [\zero, G_2] \xrightarrow{\sim} [\zero, G_1 \vee G_2].
    \qedhere \]
\end{proof}

Roughly speaking, part of the difficulty in passing from order complexes to more general nested set complexes is to understand how to pass from linear trees (chains of elements) to general trees (nested sets in built lattices). 

The rest of this section will be devoted to describing the links of nested set complexes. As mentioned in Example \ref{ex:nested-max}, in the case of the maximal building set the nested set complex is simply the order complex. In that case, for any chain $\S = F_1< \cdots < F_k$ in some lattice $\L$, the data of a chain in $\L$ containing $\S$ is equivalent to the data of a chain in each interval $[F_i, F_{i+1}],$ or put differently the link of the order complex at $\S$ decomposes as the join of the order complexes of each interval. We will see that this result holds more generally for any built lattice. We start with the following definition. 

\begin{definition}
Let $(\L, \G)$ be a built lattice, $\S$ a nested set of $(\L, \G)$ and $G$ an element of $\S$.  The \emph{local interval of $\S$ at $G$}, denoted $L^G({\S}),$ is the built lattice 
\[
L^G(\S) \coloneqq \left(\left[\bigvee \S_{< G}, G\right], \G_{\bigvee \S_{< G}}^G\right).
\]
The join of all elements in $\S$ strictly below $G$ is denoted as $J^G \coloneqq \bigvee \S_{< G}$. 
\end{definition}
\begin{example}
If $\G$ is the maximal building set, then $\S$ is some chain $G_1 < \cdots < G_k$ and $L^{G_i}(\S) = ([G_{i-1}, G_{i}], \G_{\max})$ for $1 \leqslant  i \leqslant k+1$, with the convention that $G_0 = \zero$ and $G_{k+1} = \un.$
\end{example}

The following theorem provides a decomposition of the link of a nested set complex as a join of nested set complexes of local intervals. For proofs, see \cite[Section 2.3]{coron2024matroidsfeynmancategorieskoszul} (stated for geometric lattices but hold verbatim for any lattice) and \cite[Section 2]{BDF_2022}. 

\begin{thm}\label{thm:composition-nested-sets}
    Let $(\L, \G)$ be an irreducible built lattice. 
    \begin{enumerate}[\normalfont (i)]
        \item For any $F \in \L\setminus \{\un\}$, and any $G\in \G_F$, the element $G$ has a unique $\G$-factor which is not a $\G$-factor of $F$, denoted by $G/F$.\footnote{In \cite{coron2024matroidsfeynmancategorieskoszul}, $G/F$ is denoted as $\Comp_{F}(G)$.}
        \item For any nested set $\S \in \N(\L, \G),$ and for any collection of nested sets $(\S_G)_{G \in \S \cup \un}$ with $\S_G \in \N(L^G(\S))$ for every $G \in \S\cup \un$, then the collection 
        $$ \S \circ (\S_G)_G \coloneqq \S \cup \bigcup_{G \in \S \cup \un} \{ \, G'/J^G \, | \, G' \in \S_G \, \} $$
        is a nested set of $(\L, \G)$. \footnote{The notation $\circ$ comes from the fact that this operation can be interpreted as some composition of morphisms in a certain Feynman category (see \cite[Section 3]{coron2024matroidsfeynmancategorieskoszul} for more details).}
        \item Using the same notations as above, for any $G \in \S\cup \un$, and any $G'\in \S_G$, we have an isomorphism of built lattices 
        \[
        L^{G'/J^G}(\S\circ(\S_G)_G)\simeq L^{G'}(\S_G),
        \]
        which is given by taking the join with all the factors of $G'$ which are not below $G'/J^G$. 
        \item For any nested set $\S \in \N(\L, \G)$, the map $\S \circ -$ establishes a bijection between the collections of nested sets $(\S_G)_{G\in \S \cup \{\un\}}$ where $\S_G \in \N(L^G(\S))$ for any $G \in \S \cup \un$, and the nested sets of $(\L, \G)$ containing $\S.$ In other words, we have an isomorphism of simplicial complexes 
        \[
        \lk_{\N(\L, \G)}(\S) \simeq \bigast_{G \in \S \cup \un } \N(L^G(\S)).
        \]
    \end{enumerate}
\end{thm}
In the setting of Theorem \ref{thm:composition-nested-sets} ii), if we only have a collection of nested sets $(\S_G)_{G\in I}$ for some subset $I \subset \S\cup \{\un\}$, then $\S\circ(\S_G)_G$ will mean the composition of nested sets where we implicitly set $\S_G = \emptyset$ for all $G \notin I.$

\begin{example}
    If $\G = \G_{\max}$, then for any $F$, $G/F = G$, $\S$ is a chain, the $\S_G$'s are chains in each interval of $\S,$ and $\S\circ (\S_G)_G$ is the concatenation of the chains $\S_G$, together with $\S$. The isomorphism of \Cref{thm:composition-nested-sets} (iii) is the identity.
\end{example}

The following lemma shows that the composition of nested sets is compatible with edge-labelings induced by admissible maps (see Section~\ref{subsec:admissible-lattices}). 
\begin{lemma}\label[lemma]{lem:composition-descent}
Let $(\L,\G)$ be an irreducible built lattice, together with a map $\omega:\Irr(\L)\rightarrow \Z_{>0}$, inducing the edge-labeling $\lambda_{\omega}.$ In the situation of Theorem \ref{thm:composition-nested-sets} (iii), if the local interval $L^{G'}(\S_G)$ has rank $1$, then the composition of nested sets preserves the labels, meaning that we have the equality of labels 
\[
\lambda_{\omega}(L^{G'}(\S_G) = \lambda_{\omega}( L^{G'/J^G}(\S\circ(\S_G)_G),\]
where $\lambda_{\omega}$ of a local interval of rank $1$ simply means the $\lambda_{\omega}$-label of the underlying covering relation. 
\end{lemma}

\begin{example}
In the case of the maximal building set the above result is trivial, since in that situation the two local intervals $L^{G'}(\S_G)$ and $L^{G'/J^G}(\S\circ(\S_G)_G)$ are the same.
\end{example}
\begin{proof}[Proof of \Cref{lem:composition-descent}]
From the definitions, we have the following equalities
\begin{align*}
\lambda_{\omega}(L^{G'}(\S_G)) &= \min \left\{ \omega(I) \;\middle|\; I \in \Irr(\L),\, I \vee \bigvee (\S_G)_{<G'} = G' \right\}, \\
\lambda_{\omega}(L^{G'/J^G}(\S\circ(\S_G)_G)) &= \min \left\{ \omega(I) \;\middle|\; I \in \Irr(\L),\, I \vee \bigvee (\S\circ(\S_G)_G)_{<G'/J^G} = G'/J^G \right\}.
\end{align*}
Therefore, it suffices to prove the equality of sets 
\begin{multline}\label{eq:equality-of-sets}
\left\{ I \;\middle|\; I \in \Irr(\L),\, I \vee \bigvee (\S_G)_{<G'} = G' \right\} = \\
\left\{ I \;\middle|\; I \in \Irr(\L),\, I \vee \bigvee (\S\circ(\S_G)_G)_{<G'/J^G} = G'/J^G \right\}
\end{multline} 

In one direction, if $I$ is an irreducible element of $\L$ in the left-hand side of \eqref{eq:equality-of-sets}, by Remark \ref{rmk:irreducible-in-building} we have $I \in \G$ and so since the elements $G'/J^G, G_1, \ldots, G_k$ are the factors of $G'$ we must have $I \leqslant G'/J^G$, which proves that $I$ also belongs to the right-hand side of \eqref{eq:equality-of-sets}. 

In the other direction, if we denote by $G_1,\ldots, G_k$ the factors of $G$ which are not below $G'/J^G$, then we have the equalities by Theorem \ref{thm:composition-nested-sets} (iii)
\begin{align*}
\left(G'/J^G\right) \vee G_1 \vee \cdots \vee G_k &= G', \\
\left(\bigvee (\S\circ(\S_G)_G)_{<G'/J^G}\right) \vee G_1 \vee \cdots \vee G_k &= \bigvee (\S_G)_{<G'}.
\end{align*}
These inequalities imply the desired inclusion. 
\end{proof}

\begin{example}
Consider $\mathcal{B}_3$ the Boolean lattice over $3$ elements $1,2,3$, with building set $\G = \{1, 2 ,3, 123\}.$ The element $12$ is not in $\G$, but it belongs to the building set of the local interval $L^{\un}(\{1\})$, so the composition $\S = \{1\}\circ \{12\}$ makes sense. The unique factor of $12$ which is not $1$ is $2$, and so we have $\S = \{1, 2\}$. The local intervals of $\S$ are 
\begin{align*}
    L^1(\S)&= ([\zero, 1], \G_{\max}),\\
    L^2(\S) &= ([\zero, 2], \G_{\max}),\\
    L^{\un}(\S) &= ([12, \un], \G_{\max}).
\end{align*}
The local intervals of the nested set $\{12\}$ in the local interval $L^{\un}(\{1\})$ are $([1, 12], \G_{\max})$ and $([12, \un],\G_{\max})$. By Theorem \ref{thm:composition-nested-sets} (iii), taking the join with $1$ gives an isomorphism
\[
L^2(\S) \simeq ([12, \un],\G_{\max}). 
\]
The irreducible elements of $\mathcal{B}_3$ are the atoms $1,2,3$, and we can define $\omega:\Irr(\mathcal{B}_3) \rightarrow \Z_{>0}$ by $\omega(i) = i$ for $i \in \{1,2,3\}$. \Cref{lem:composition-descent} implies that 
\[
\lambda_{\omega}(1,12) =  \lomeg(\zero,2) = 2.
\]
\end{example}

One particular case of composition of nested sets which will be of special interest to us is the case when maximal chains yield maximal nested sets. Let $(\L, \G)$ be a built lattice, and let $\zero \lessdot F_1 \lessdot F_2 \lessdot \cdots \lessdot F_n = \un$ be a maximal chain in $\L$. Since each $F_i$ is an atom in $[F_{i-1}, \un]$, and a building set must contain every atom by Remark \ref{rmk:irreducible-in-building}, we have that each $F_i$ belongs to the building set $\G_{F_{i-1}}$. This means that the composition $\{F_1\}\circ \{F_2\}\circ \cdots \circ \{F_{n-1}\}$ is a well-defined nested set, even if the $F_i$'s need not belong to $\G$. By Theorem \ref{thm:composition-nested-sets} (iii), all the local intervals of that nested set have rank $1,$ and there are exactly $n$ of them (including the local interval at $\un$). In fact, we have the following lemma, which shows that this property characterizes maximal nested sets in general. 
\begin{lemma}\label{lem:max-nested-purity}    
Let $(\L, \G)$ be a built lattice. A nested set $\S \in \N(\L, \G)$ is maximal if and only if all its local intervals $L^G(\S)$ for $G \in \S \cup \un$ have rank $1$. If $\L$ is graded of rank $n$, then the nested set complex $\N(\L, \G)$ is pure of dimension $n-2.$
\end{lemma}

\begin{proof}
The first statement is a direct consequence of Theorem \ref{thm:composition-nested-sets} (iv): the simplices of $\N(\L, \S)$ containing $\S$ are in bijection with the simplices of the join 
\[
\lk_{\N(\L, \G)}(\S) \simeq \bigast_{G \in \S \cup \un } \N(L^G(\S))
\] which has a nontrivial simplex if and only if at least one of the local intervals has rank at least $2$. 

For the second statement, we prove the result by induction on the rank $n$ of $\L$. Let $\S \in \N(\L, \G)$ be a maximal nested set and let $G$ be any minimal element in $\S.$ By the previous statement, $G$ must be an atom. By Theorem \ref{thm:composition-nested-sets} (iv), we have $\S = \{G\}\circ \S'$ for some $\S' \in \N(\L_G, \G_G)$. Also by Theorem \ref{thm:composition-nested-sets} (iv), the nested set $\S'$ is maximal. Since $\L_G$ is graded of rank $n-1$, we have that $\abs{\S'} = n-2$ by induction. From the definition of the composition of nested sets, we see that this implies that  $\S = \{G\}\circ \S'$ has cardinality $n-1$ which concludes the proof. 
\end{proof}

We end this subsection with an example showing that the class of nested set complexes of geometric lattices contains the augmented Bergman complexes of matroids. 

\begin{example}\label{ex:augmented-bergman}
    The following example is a consequence of an observation attributed to Chris Eur and written in detail by Mastroeni and McCullough in \cite[Section~5.1]{mastroeni-mccullough}, we refer to this paper for more details. Consider a matroid $\M$, and let $\M'$ be the \emph{free coextension} of $\M$, i.e., the matroid that results from the operation $\trunc(\M^*\oplus \mathrm{U}_{1,1})^*$, where $\trunc$ stands for truncation and the asterisk for matroid duality; we refer to \cite{oxley} for more details about it. Let us denote by $e$ the single element that we added to the ground set of $\M$ to obtain $\M'$. The flats of $\M'$ correspond to i) sets of the form $F\cup e$ where $F$ is a flat of $\M$, and ii) independent sets of $\M$. The subset of $\mathcal{L}(\M')$ consisting of the flats of the form $F\cup e$ for $F\in \mathcal{L}(\M)$ is a building set on $\L(\M')$, denoted $\G_{\aug}$. The nested set complex $\N(\L(\M'),\mathcal{G}_{\aug})$ is isomorphic to the \emph{augmented Bergman complex} of $\M$, introduced by Braden, Huh, Matherne, Proudfoot, and Wang in \cite{braden-huh-matherne-proudfoot-wang}. 
\end{example}

\section{Vertex decomposability}\label{sec:vert-decomp}

The main result of this section is that for any built geometric lattice $(\L, \G)$, the nested set complex $\N(\L,\G)$ is vertex decomposable. More generally, we will prove that for any graded built lattice $(\L, \G)$, if $\L$ admits an injective admissible map, then $\N(\L, \G)$ is vertex decomposable (Theorem~\ref{thm:vertex decomposability}). This will allow us to derive an explicit combinatorial formula for the $h$-polynomial of the nested set complex in that case (Proposition \ref{prop:h-poly-descents}). In the case of the maximal building set, this vertex decomposition is known by the fact that admissible lattices are EL-shellable \cite[Section 3]{bjorner-el-shellable}, and the fact that order complexes of EL-shellable posets are vertex decomposable \cite[Theorem~11.6]{BjornerWachs1997}.

\subsection{The vertex decomposition}
The main difficulty we have to face in order to prove the vertex decomposition amounts to understanding the deletion $\N(\L, \G) \setminus G$ of a nested set complex $\N(\L, \G)$ for some particular $G\in \G$. More specifically, for inductive purposes we would like to express this deletion as the nested set complex of some smaller built lattice. This is addressed by the following definitions and lemmas.  

\begin{definition}
    Let $(\L, \G)$ be a built lattice and let $G$ be an element of $\G$. The \emph{building ideal} generated by $G$, denoted by $(G)$, is defined as the subset 
    \begin{equation*}
         (G) \coloneqq \{ F \in \L \, | \, \text{$G$ is a  $\G$-factor of $F$}\} \subseteq \L.
    \end{equation*}
\end{definition}
For example, if $\G$ is the maximal building set, then for any $G \in \G$, $(G) = \{G\}$. 

The following two lemmas show that removing a join-irreducible element $G$ from a built lattice $(\L, \G)$ yields a smaller built lattice $(\L\setminus (G), \G \setminus G)$. 
\begin{lemma}\label{lem:restriction-lattice}
    For any built lattice $(\L, \G)$ and any $G \in \G$, if $G$ is join-irreducible then $\L\setminus (G)$ is a lattice.   
\end{lemma}

\begin{proof}
    Since $\zero \notin (G)$, the poset $\L\setminus (G)$ has a lower bound and so it is enough to prove that $\L$ is a join-semilattice. Let $A$ and $B$ be two elements of $\L \setminus (G),$ and let $G_1, \ldots, G_k$ be the factors of $A\vee B$, where the join is taken in $\L.$ If none of the factors are equal to $G$, then we have $A\vee B \in \L\setminus (G)$ and so $A \vee B $ is the supremum of $A$ and $B$ in $\L\setminus (G).$ On the other hand, if $G$ is one of the factors of $A\vee B$, say $G_1$, then by the structural isomorphism of $\G$ at $A \vee B$
    \[
    \psi_{A \vee B} \colon [\zero, A \vee B] \simeq \prod_{1\leqslant i \leqslant k }[\zero, G_i], 
    \] we can uniquely write $A = \bigvee_i A_i$ and $B = \bigvee_i B_i$ for $A_i, B_i$ with $1 \leqslant i \leqslant k$ some elements below $G_i$ for $1 \leqslant i \leqslant k$. By the same structural isomorphism, we have $A_1 \vee B_1 = G_1 = G$. Since $G$ is irreducible, we must have $A_1 =G$ or $B_1 = G$. By Lemma \ref{lem:factor-union} this implies that either $A$ or $B$ has $G$ as one of its factors, which is a contradiction. 
\end{proof}

\begin{lemma}
    For any built lattice $(\L, \G)$ and any join-irreducible $G\in \G$, the set $\G \setminus G$ is a building set of the lattice $\L \setminus (G)$.
\end{lemma}
This result is true even if $G$ is not assumed join-irreducible, provided one works with a definition of building set for meet-semilattices rather than just lattices; see \cite[Definition 2.2]{feichtner-kozlov}.

\begin{proof}
    In this proof, for any elements $A\leqslant B \in \L \setminus (G)$, the notation $[A,B]$ will mean the interval in $\L$ and $[A,B]_G$ will mean the interval in $\L \setminus (G).$ Let $F$ be an element of $\L\setminus (G)$, with $\G$-factors $G_1, \ldots, G_k.$ If $G\nleqslant F$, then the $(\G\setminus G)$-factors of $F$ are the $\G$-factors of $F$, and we have the isomorphisms of posets
    \begin{equation*}
        [\zero, F]_G = [\zero, F] \simeq \prod_{1 \leqslant i \leqslant k}[\zero, G_i] = \prod_{1 \leqslant i \leqslant k}[\zero, G_i]_G.
    \end{equation*}
    On the other hand, if $G\leqslant F$, then $G$ is below one of the factors of $F,$ say $G_1.$ Since $F$ does not belong to $(G)$, we have $G_1 \neq G$ and so the $(\G\setminus G)$-factors of $F$ are the $\G$-factors of $F$. Moreover, notice that for any collection of elements $F_i \leqslant G_i$ for $1 \leqslant  i  \leqslant k$, by Lemma \ref{lem:factor-union} we have that $\bigvee_i F_i \in (G)$ if and only if $F_1 \in (G)$. This gives the isomorphisms of posets
    \begin{equation*}
        [\zero, F]_G \simeq [\zero, G_1]_G\times \prod_{1 <i \leqslant k}[\zero, G_i] = \prod_{1\leqslant i \leqslant k}[\zero, G_i]_G. \qedhere
    \end{equation*}
\end{proof}
Figure \ref{fig:building-deletion} shows the deletion of a building ideal in the Boolean lattice over three elements. The elements of the considered building set are circled. Since we are deleting an atom, the poset stays a lattice.
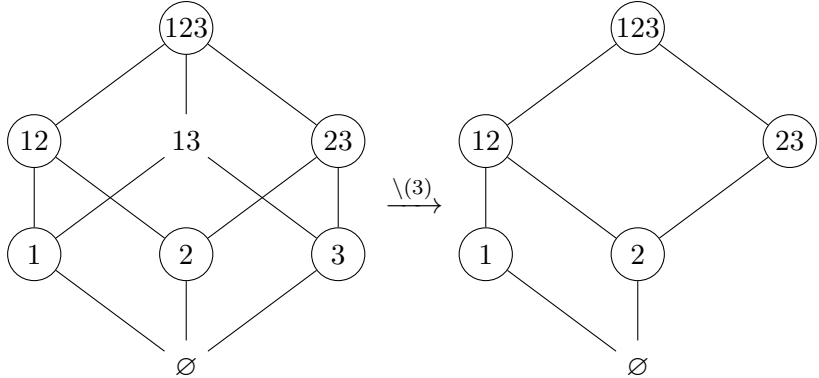
\begin{figure}[H]
    \centering
\begin{tikzpicture}[
  node/.style = {circle, draw, minimum size=7mm, inner sep=0pt},
  tnode/.style = {circle, draw = none, minimum size=7mm, inner sep=0pt}
]

\node[tnode] (0)   at (0,0)   {$\emptyset$};

\node[node] (1)   at (-2,1.5) {$1$};
\node[node] (2)   at (0,1.5)  {$2$};
\node[node] (3)   at (2,1.5)  {$3$};

\node[node] (12)  at (-2,3)   {$12$};
\node[tnode] (13)  at (0,3)    {$13$};
\node[node] (23)  at (2,3)    {$23$};

\node[node] (123) at (0,4.5)  {$123$};
\node[tnode] (arrow) at (3,2.3) {$\xrightarrow{\setminus(3)}$};

\draw (0) -- (1);
\draw (0) -- (2);
\draw (0) -- (3);

\draw (1) -- (12);
\draw (1) -- (13);

\draw (2) -- (12);
\draw (2) -- (23);

\draw (3) -- (13);
\draw (3) -- (23);

\draw (12) -- (123);
\draw (13) -- (123);
\draw (23) -- (123);
\end{tikzpicture}
\begin{tikzpicture}[
  node/.style = {circle, draw, minimum size=7mm, inner sep=0pt},
  tnode/.style = {circle, draw = none, minimum size=7mm, inner sep=0pt}
]

\node[tnode] (0)   at (0,0)   {$\emptyset$};

\node[node] (1)   at (-2,1.5) {$1$};
\node[node] (2)   at (0,1.5)  {$2$};

\node[node] (12)  at (-2,3)   {$12$};

\node[node] (23)  at (2,3)    {$23$};

\node[node] (123) at (0,4.5)  {$123$};

\draw (0) -- (1);
\draw (0) -- (2);

\draw (1) -- (12);

\draw (2) -- (12);
\draw (2) -- (23);

\draw (12) -- (123);

\draw (23) -- (123);
\end{tikzpicture}
\caption{Deletion of a building ideal}
    \label{fig:building-deletion}
\end{figure}

Finally, we have the following lemma, which answers the question asked in the preamble of this section. 

\begin{lemma}\label{lem:restriction-of-nested-cx}
    For any built lattice $(\L, \G)$ and for any join-irreducible element $G\in \G$, we have the equality of nested set complexes 
    \[ \N(\L, \G) \setminus G = \N(\L\setminus (G), \G\setminus G).\]
\end{lemma}

\begin{proof}
    For any subset $\S \subset \G\setminus G$ and any antichain $\mathcal{A}\subset \S$ of size at least $2$, $\bigvee \mathcal{A} \in \G\setminus G $ if and only if $\bigvee \mathcal{A} \in \G$ (where $\bigvee$ means both the join in $\L$ and in $\L \setminus (G)$ since those joins are equal). Therefore, $\S$ is a nested set of $\N(\L, \G)$ not containing $G$ if and only if $\S$ is a nested set of $\N(\L\setminus (G), \G\setminus G).$ 
\end{proof}

Note that in the above lemma the assumption $G\in \G$ is superfluous, as a building set must contain every join-irreducible element (Remark \ref{rmk:irreducible-in-building}). We are now ready to prove the main theorem of this section. 

\newtheorem*{thm:intro1}{Theorem~\ref{thm:vertex decomposability}}
\begin{thm:intro1}
    Let $(\L, \G)$ be a graded built lattice. If $\L$ admits an injective admissible map, then the nested set complex $\N(\L, \G)$ is vertex decomposable. 
\end{thm:intro1}

\begin{proof}
    The proof goes by induction on the cardinality of $\L$. By Lemma \ref{lem:reducible} we can assume that $(\L, \G)$ is irreducible. Let $\omega: \Irr(\L) \rightarrow \Z_{\geq 0 }$ be an injective admissible map on $\L$. Let $G_{\max}\in \G$ be the unique irreducible element which maximizes $\omega$. If $G_{\max} = \un$, this means that $\un$ covers a unique element $F$, and so we have $\N(\L, \G) \simeq c\N(\L_F, \G_F)$. In that case, Lemma \ref{lem:reducible} allows us to conclude by induction. 
    
    We now assume $G_{\max} \neq \un$. By the results of Section \ref{sec:prelim}, we have the isomorphism of simplicial complexes 
    \begin{equation*}
        \lk_{\N(\L,\G)}(G_{\max}) \simeq \N(\L_{G_{\max}}, \G_{G_{\max}})*\N(\L^{G_{\max}}, \G^{\G_{\max}}),
    \end{equation*}
    which is vertex decomposable by induction, since a join of two vertex decomposable simplicial complexes is vertex decomposable. By Lemma \ref{lem:restriction-of-nested-cx} we have the equality of simplicial complexes
    \begin{equation*}
        \N(\L, \G)\setminus G_{\max} = \N(\L\setminus (G_{\max}), \G\setminus G_{\max}).
    \end{equation*}
    To conclude by induction that $\N(\L, \G)\setminus G_{\max}$ is pure vertex decomposable of same dimension as $\N(\L, \G)$, by Lemma \ref{lem:max-nested-purity} it is enough to prove that the lattice $\L\setminus (G_{\max})$ is graded of same rank as $\L$ and admits an injective admissible map. We start with the following claim.
    \begin{claim}\label{lem:chain-in-restriction}
    For any $A \leqslant B \in \L \setminus (G_{\max})$, the unique saturated chain in $\L$ with increasing $\lambda_{\omega}$-labels between $A$ and $B$ is contained in $\L\setminus (G_{\max}).$
    \end{claim} 
    \noindent \emph{Proof.} Let $A_0 = A \lessdot A_1 \lessdot \cdots \lessdot A_k = B$ be the unique saturated chain with increasing $\lambda_{\omega}$-labels between $A$ and $B$. Suppose for contradiction that $i$ is the smallest index such that $A_i \in (G_{\max})$. If we denote the factors of $A_i$ by $G_1 = G_{\max}, G_2, \ldots, G_p$, then we have the isomorphism in the definition of building set, $[\zero, A_i] \simeq \prod_{1 \leqslant q \leqslant p}[\zero, G_q]$ and the element $A_{i-1}$ can be uniquely written as a join $A_{i-1} = \bigvee_{1\leqslant q \leqslant p} A_{i-1}^{q}$ where each $A^q_{i-1}$ is below $G_q$ respectively. By minimality of $i$, we have $A^1_{i-1} < G_{\max}$ and so by the same structural isomorphism, we must have $\lambda_{\omega}(A_{i-1}, A_i) = \omega(G_{\max}).$ Since $G_{\max}\leqslant A_{i}$ we must have $\lambda_{\omega}(A_i, A_{i+1}) \neq \omega(G_{\max}),$ and so by maximality of $\omega(G_{\max})$ we have that $\lambda_{\omega}$ descends at $A_{i}$, which is a contradiction. 

\medskip

This proves that $\L\setminus (G)$ is graded, with rank given by the restriction of the rank of $\L.$ We now turn toward proving that $\L\setminus (G)$ admits an injective admissible map. One cannot simply restrict $\omega$, as deleting an irreducible element in a lattice can create new irreducible elements. Let us denote by $G_1, \ldots, G_k \in I(\L\setminus (G_{\max})) \setminus \Irr(\L)$ the new irreducible elements in any order. We define $\tilde{\omega}:I(\L\setminus (G))\rightarrow \Z_{\geqslant 0}$ by setting $\tilde{\omega}(G) = \omega(G)$ for any $G \in \Irr(\L) \cap I(\L\setminus (G))$, and $\tilde{\omega}(G_i) = \omega(G_{\max}) + i$ for $1 \leqslant i \leqslant k.$ By construction the map $\tilde{\omega}$ is injective. What remains is proving that it is admissible. This will be a consequence of the following claim. 
\begin{claim}\label{lem:labels-of-restriction}
For any $A < B \in \L\setminus (G_{\max})$, $B$ covers $A$ in $\L \setminus (G_{\max})$ if and only if $B$ covers $A$ in $\L.$ Moreover in that case, we have that $\lambda_{\omega}(A,B) < \omega(G_{\max})$ if and only if $\lambda_{\Tilde{\omega}}(A,B) < \omega(G_{\max})$ and in that situation we have $\lambda_{\omega}(A, B)= \lambda_{\Tilde{\omega}}(A,B).$
\end{claim}
\noindent \emph{Proof.} The first statement is a consequence of Claim~\ref{lem:chain-in-restriction}. If $G$ is an irreducible element of $\L$ different from $G_{\max}$, then $G$ belongs to $\L\setminus (G_{\max})$, and we have $A\vee G = B$ in $\L$ if and only if $A\vee G = B$ in $\L\setminus (G_{\max})$, since $\L$ and $\L\setminus (G_{\max})$ have the same join (see Lemma \ref{lem:restriction-lattice}). This proves the two last statements.

\medskip

We are now ready to prove that $\Tilde{\omega}$ is an admissible map. For any $A \leqslant B \in \L\setminus (G)$ by Claim~\ref{lem:chain-in-restriction} the unique saturated chain between $A$ and $B$ with increasing $\lambda_{\omega}$-labels belongs to $\L\setminus (G_{\max}).$ By the proof of Claim~\ref{lem:chain-in-restriction} the $\lambda_{\omega}$-labels of that saturated chain are all strictly less than $\omega(G_{\max})$ except possibly the last one, and so by Claim~\ref{lem:labels-of-restriction} the $\lambda_{\Tilde{\omega}}$-labels are the same, except possibly the last one. In the case that the last labels differ, this means that the last label is of the form $\omega(G_{\max}) + i$ which is greater than all the other labels. In conclusion that saturated chain has increasing $\lambda_{\Tilde{\omega}}$-labels. We have to show that this is the only one. Let $A_0 = A \lessdot A_1 \lessdot \cdots \lessdot A_k = B$ be some saturated chain in $\L\setminus (G)$ with increasing $\lambda_{\Tilde{\omega}}$-labels. By Lemma \ref{lem:labels-of-restriction} the $\lambda_{\omega}$-labels of that chain are obtained from the $\lambda_{\Tilde{\omega}}$-labels by replacing every label greater than $\omega(G_{\max})$ by $\omega(G_{\max}).$ This is easily seen to be weakly increasing as soon as the $\lambda_{\Tilde{\omega}}$-labels are weakly increasing. 
\end{proof}

We include Figure \ref{fig:successive-deletions} which illustrates the successive deletions of building ideals dictated by the proof of Theorem \ref{thm:cvx-ear-dec}. The starting lattice is the Boolean lattice of rank $4$, and the elements of the building sets are circled in each lattice. The value of the admissible map at each irreducible element is written next to the circle in red. In the end the process stops because we arrive at a simplex. 

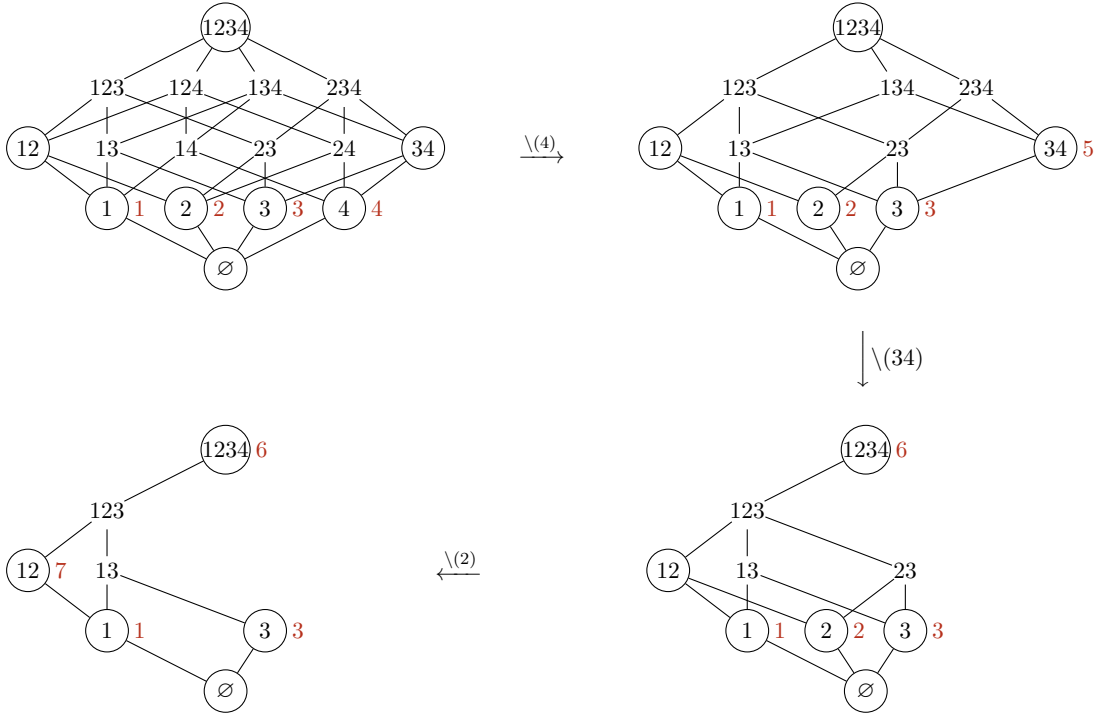
\begin{figure}
  \centering

\begin{tikzpicture}[
    scale=0.8,
    transform shape,
    x=0.65cm,
    node/.style = {circle, draw, minimum size=20pt, inner sep=0pt},
    tnode/.style = {circle, draw = none, minimum size=5pt, inner sep=0pt}
]

\begin{scope}[shift={(0,0)}]

  \node[node] (e1)   at (0,0)   {$\emptyset$};

  \node[node] (a1)   at (-3,1)  {$1$};
  \node[node] (b1)   at (-1,1)  {$2$};
  \node[node] (c1)   at ( 1,1)  {$3$};
  \node[node] (d1)   at ( 3,1)  {$4$};

  \node[BrickRed, right=-0.5mm of a1] (a1right) {$1$};
  \node[BrickRed, right=-0.5mm of b1] (b1right) {$2$};
  \node[BrickRed, right=-0.5mm of c1] (c1right) {$3$};
  \node[BrickRed, right=-0.5mm of d1] (d1right) {$4$};

  \node[node]  (ab1) at (-5,2)  {$12$};
  \node[tnode] (ac1) at (-3,2)  {$13$};
  \node[tnode] (ad1) at (-1,2)  {$14$};
  \node[tnode] (bc1) at ( 1,2)  {$23$};
  \node[tnode] (bd1) at ( 3,2)  {$24$};
  \node[node]  (cd1) at ( 5,2)  {$34$};

  \node[tnode] (abc1) at (-3,3)  {$123$};
  \node[tnode] (abd1) at (-1,3)  {$124$};
  \node[tnode] (acd1) at ( 1,3)  {$134$};
  \node[tnode] (bcd1) at ( 3,3)  {$234$};

  \node[node] (abcd1) at (0,4)  {$1234$};

  \draw (e1) -- (a1) (e1) -- (b1) (e1) -- (c1) (e1) -- (d1);

  \draw (a1) -- (ab1) (a1) -- (ac1) (a1) -- (ad1);
  \draw (b1) -- (ab1) (b1) -- (bc1) (b1) -- (bd1);
  \draw (c1) -- (ac1) (c1) -- (bc1) (c1) -- (cd1);
  \draw (d1) -- (ad1) (d1) -- (bd1) (d1) -- (cd1);

  \draw (ab1) -- (abc1) (ab1) -- (abd1);
  \draw (ac1) -- (abc1) (ac1) -- (acd1);
  \draw (ad1) -- (abd1) (ad1) -- (acd1);
  \draw (bc1) -- (abc1) (bc1) -- (bcd1);
  \draw (bd1) -- (abd1) (bd1) -- (bcd1);
  \draw (cd1) -- (acd1) (cd1) -- (bcd1);

  \draw (abc1) -- (abcd1)
        (abd1) -- (abcd1)
        (acd1) -- (abcd1)
        (bcd1) -- (abcd1);

\end{scope}

\node[tnode] at (8,2) {$\xrightarrow{\setminus (4)}$};

\begin{scope}[shift={(16,0)}]

  \node[node] (e2)   at (0,0)   {$\emptyset$};

  \node[node] (a2)   at (-3,1)  {$1$};
  \node[node] (b2)   at (-1,1)  {$2$};
  \node[node] (c2)   at ( 1,1)  {$3$};

  \node[BrickRed, right=-0.5mm of a2] (a2right) {$1$};
  \node[BrickRed, right=-0.5mm of b2] (b2right) {$2$};
  \node[BrickRed, right=-0.5mm of c2] (c2right) {$3$};

  \node[node]  (ab2) at (-5,2)  {$12$};
  \node[tnode] (ac2) at (-3,2)  {$13$};
  \node[tnode] (bc2) at ( 1,2)  {$23$};
  \node[node]  (cd2) at ( 5,2)  {$34$};

  \node[BrickRed, right=-0.5mm of cd2] (cd2right) {$5$};

  \node[tnode] (abc2) at (-3,3) {$123$};
  \node[tnode] (acd2) at ( 1,3) {$134$};
  \node[tnode] (bcd2) at ( 3,3) {$234$};

  \node[node] (abcd2) at (0,4) {$1234$};

  \draw (e2) -- (a2) (e2) -- (b2) (e2) -- (c2);

  \draw (a2) -- (ab2) (a2) -- (ac2);
  \draw (b2) -- (ab2) (b2) -- (bc2);
  \draw (c2) -- (ac2) (c2) -- (bc2) (c2) -- (cd2);

  \draw (ab2) -- (abc2);
  \draw (ac2) -- (abc2) (ac2) -- (acd2);
  \draw (bc2) -- (abc2) (bc2) -- (bcd2);
  \draw (cd2) -- (acd2) (cd2) -- (bcd2);

  \draw (abc2) -- (abcd2)
        (acd2) -- (abcd2)
        (bcd2) -- (abcd2);

\end{scope}

\node[tnode] at (16.8,-1.5) {$\bigg\downarrow{\setminus (34)}$};

\begin{scope}[shift={(16.2,-7)}]

  \node[node] (e4)   at (0,0)   {$\emptyset$};

  \node[node] (a4)   at (-3,1)  {$1$};
  \node[node] (b4)   at (-1,1)  {$2$};
  \node[node] (c4)   at ( 1,1)  {$3$};

  \node[BrickRed, right=-0.5mm of a4] (a2right) {$1$};
  \node[BrickRed, right=-0.5mm of b4] (b2right) {$2$};
  \node[BrickRed, right=-0.5mm of c4] (c2right) {$3$};

  \node[node]  (ab4) at (-5,2)  {$12$};
  \node[tnode] (ac4) at (-3,2)  {$13$};
  \node[tnode] (bc4) at ( 1,2)  {$23$};

  \node[tnode] (abc4) at (-3,3) {$123$};

  \node[node] (abcd4) at (0,4)  {$1234$};
  \node[BrickRed, right=-0.5mm of abcd4] (abcd4right) {$6$};

  \draw (e4) -- (a4) (e4) -- (b4) (e4) -- (c4);
  \draw (a4) -- (ab4) (a4) -- (ac4);
  \draw (b4) -- (ab4) (b4) -- (bc4);
  \draw (c4) -- (ac4) (c4) -- (bc4);
  \draw (ab4) -- (abc4) (ac4) -- (abc4) (bc4) -- (abc4);
  \draw (abc4) -- (abcd4);

\end{scope}

\node[tnode] at (5.9,-4.9) {$\xleftarrow{\setminus (2)}$};

\begin{scope}[shift={(0,-7)}]

  \node[node] (e3)   at (0,0)   {$\emptyset$};

  \node[node] (a3)   at (-3,1)  {$1$};
  \node[node] (c3)   at ( 1,1)  {$3$};

  \node[node]  (ab3) at (-5,2)  {$12$};
  \node[tnode] (ac3) at (-3,2)  {$13$};

  \node[tnode] (abc3) at (-3,3) {$123$};

  \node[node] (abcd3) at (0,4)  {$1234$};

  \node[BrickRed, right=-0.5mm of a3] (a3right) {$1$};
  \node[BrickRed, right=-0.5mm of c3] (c3right) {$3$};
  \node[BrickRed, right=-0.5mm of ab3] (ab3right) {$7$};
  \node[BrickRed, right=-0.5mm of abcd3] (abcd3right) {$6$};
  \draw (e3) -- (a3) (e3) -- (c3);
  \draw (a3) -- (ab3) (a3) -- (ac3);
  \draw (c3) -- (ac3);
  \draw (ab3) -- (abc3) (ac3) -- (abc3);
  \draw (abc3) -- (abcd3);

\end{scope}

\end{tikzpicture}
  \caption{Successive deletions of building ideals}
  \label{fig:successive-deletions}
\end{figure}

\subsection{A formula for the \texorpdfstring{$h$}{h}-vector of a nested set complex}\label{subsec:h-polynomial}
In this subsection we give an explicit formula for the $h$-vector of any nested set complex associated to a graded built lattice with an injective admissible map. This formula generalizes the formula for the $h$-vector of a CL-shellable poset in terms of counting maximal chains with the descent statistics (see \cite[Theorem 3.14.2]{stanley-ec1}). Our first task is to make sense of the descent number of a maximal nested set. Throughout this section, we let $(\L, \G)$ be an irreducible graded built lattice, such that $\L$ admits an injective admissible map $\omega$, with associated labeling $\lambda_{\omega}$. The $h$-polynomial of $\N(\L , \G)$ is denoted $h(\L, \G)(t).$ By Lemma \ref{lem:nested-tree}, any element $G$ in some nested set $\S \in \N(\L, \G)$ has a unique parent in $\S\cup \{\un\}$, which is denoted by $p(G) \coloneqq \min (\S\cup \un)_{> G} $. 

\begin{definition}
Let $\S \in \N(\L, \G)$ be a maximal nested set. For any $G \in \S \cup \un$, we denote $\lambda_{\omega}(G) = \lambda_{\omega}(\bigvee \S_{<G},G)$. We say that an element $G \in \S$ is a \textit{descent} of $\S$ if we have the strict inequality $\lomeg(G) > \lomeg(p(G)).$ The \textit{descent number} of $\S$, denoted $\des(\S)$, is the cardinality 
$$ \des(\S) \coloneqq |\{ \, G \, | \, G \textrm{ is a descent of } \S \, \}|.$$ The \textit{descent $h$-polynomial} of $(\L, \G)$ is 
$$ h_{\des}(\L,\G)(t) \coloneqq \sum_{\substack{\S \in \N(\L, \G) \\ \S \textrm{ maximal}}} t^{\des(\S)}.$$
\end{definition}
Note that Lemma \ref{lem:max-nested-purity} ensures that all the local intervals of $\S$ have rank $1$ and so $\lambda(G)$ is well-defined for any $G \in \S\cup\un.$ The main result of this subsection is the following. 

\begin{proposition}\label{prop:h-poly-descents}
    We have the equality of polynomials 
    \[ h_{\des}(\L, \G)(t)= h(\L, \G)(t).\]
\end{proposition}

\begin{proof}
    The proof goes by induction on the cardinality of $\L$. As in the proof of Theorem \ref{thm:vertex decomposability}, let us denote by $G_{\max}$ the unique join-irreducible element which maximizes $\omega.$ If $G_{\max} = \un$ then $G_{\max}$ covers a unique element $F$ and we have the identification $\N(\L, \G) \simeq c\N(\L_F,\G_F)$. By maximality of $\omega(G_{\max})$ this identification preserves the descent statistics, and so we can conclude by induction using Lemma \ref{lem:reducible}. If $G_{\max} \neq \un$ then the proof of Theorem \ref{thm:vertex decomposability} shows that $G_{\max}$ is a shedding vertex. This implies the recursive equality for the face numbers 
    \[ f_i(\N(\L, \G)) = f_i(\N(\L, \G)\setminus G_{\max}) + f_{i-1}(\lk_{\N(\L, \G)}(G_{\max})) \]
    for every $i$, which implies the equalities of polynomials 
    \begin{align*} 
    h(\L, \G)(t) &= h(\L\setminus (G_{\max}), \G\setminus G_{\max})(t) + t\cdot  h(\lk_{\N(\L, \G)}(G_{\max}))(t)\\
    &= h(\L\setminus (G_{\max}), \G\setminus G_{\max})(t) + t \cdot h(\L_{G_{\max}}, \G_{G_{\max}})(t) \cdot h(\L^{G_{\max}}, \G^{G_{\max}})(t),
    \end{align*}
    where the second equality comes from Theorem \ref{thm:composition-nested-sets}. It remains to prove that $h_{\des}$ satisfies the same recursive relation. We have 
    \begin{align*}
    h_{\des}(\L, \G)(t) &= \sum_{\substack{\S \in \N(\L, \G) \\ \S \textrm{ maximal}}} t^{\des(\S)} = \sum_{\substack{\S \in \N(\L, \G) \\ \S \textrm{ maximal} \\ G_{\max} \notin \S}} t^{\des(\S)} + \sum_{\substack{\S \in \N(\L, \G) \\ \S \textrm{ maximal}\\ G_{\max}\in \S}} t^{\des(\S)} \\ 
    &= h_{\des}(\L\setminus(G_{\max}),\G\setminus G_{\max})(t) \\
    &\qquad+ t \cdot h_{\des}(\L_{G_{\max}}, \G_{G_{\max}})(t)\cdot h_{\des}(\L^{G_{\max}}, \G^{G_{\max}})(t),
    \end{align*}
    where the last equality comes from Lemma \ref{lem:restriction-of-nested-cx}, together with the isomorphism in Theorem \ref{thm:composition-nested-sets} (iv):
    \[\lk_{\N(\L,\G)}(G_{\max}) \simeq \N(\L_{G_{\max}}, \G_{G_{\max}})*\N(\L^{G_{\max}}, \G^{G_{\max}})\] which preserves the descent statistics by Lemma \ref{lem:composition-descent}, and the fact that adding the element $G_{\max}$ to a nested set in the link at $G_{\max}$ always adds one descent by maximality of $\omega(G_{\max}).$
\end{proof}

\begin{example}
Consider $\mathcal{B}_{3}$ the Boolean lattice on three elements $\{1,2,3\}$, with the building set $\G = \{1,2,3,12, 23, 123\}.$ The built lattice $(\mathcal{B}_3, \G)$ has five maximal nested sets, namely 
\begin{align*}
    \S_1 = \{1,12\}, \S_2 = \{2,12\}, \S_3 = \{2,23\}, \S_4 = \{3,23\}, \textrm{ and }\S_5 = \{1,3\}. 
\end{align*}
In $\S_1$, there is no descent so $\des(\S_1) = 0.$ In $\S_2$, only $2$ is a descent, so $\des(\S_2) = 1.$ In $\S_3$, only $23$ is a descent so $\des(\S_3) = 1.$ In $\S_4,$ both $3$ and $23$ are descents so $\des(\S_4) = 2.$ Finally, in $\S_5$, only $3$ is a descent so $\des(\S_5) = 1.$ Summing those contributions we obtain 
$$ h(\mathcal{B}_3, \G)(t) = 1 + 3t + t^2.$$
\end{example}

\begin{remark}
More generally, it is known that any shelling order of a simplicial complex gives a formula for the corresponding $h$-polynomial in terms of counting facets with the restriction statistics (see \cite[Section 2]{bjorner} for instance). It would be interesting to compare the formula obtained this way via the shelling of nested set complexes constructed in \cite{backman-dorpalen-nathanson-partida-prime}, and the formula given by Proposition \ref{prop:h-poly-descents}.
\end{remark}

\section{Convex ear decomposition}\label{sec:conv-ear-dec}
In this section we prove that for any built geometric lattice $(\L, \G)$, the nested set complex $\N(\L,\G)$ admits a convex ear decomposition. Contrary to the previous section, for this result it is necessary to restrict to the setting of matroids/geometric lattices. We refer to \cite{oxley} for any undefined matroid-related terminology. The convex ear decomposition that we will give generalizes that given in \cite[Section~4]{nyman-swartz} in the case of the maximal building set.

\begin{figure}[h!]
\begin{tikzpicture}[
    purplevertex/.style={circle, fill=purpledot, minimum size=8pt, inner sep=0pt},
    orangevertex/.style={circle, fill=orangedot, minimum size=8pt, inner sep=0pt},
    tealvertex/.style={circle, fill=tealdot, minimum size=8pt, inner sep=0pt},
    purpleline/.style={line width=2pt, purpleedge},
    orangeline/.style={line width=2pt, orangeedge},
    tealline/.style={line width=2pt, tealedge},
    label/.style={font=\large\bfseries},
    scale=1
]

\coordinate (v1) at (90:2.5);      
\coordinate (v12) at (30:2.5);     
\coordinate (v2) at (-30:2.5);     
\coordinate (v23) at (-90:2.5);    
\coordinate (v3) at (-150:2.5);    
\coordinate (v13) at (150:2.5);    

\coordinate (v4) at (0, 0);        
\coordinate (v14) at (90:1.2);     
\coordinate (v24) at (-30:1.2);    
\coordinate (v34) at (-150:1.2);   

\draw[purpleline] (v1) -- (v12);
\draw[purpleline] (v12) -- (v2);
\draw[purpleline] (v2) -- (v23);
\draw[purpleline] (v23) -- (v3);
\draw[purpleline] (v3) -- (v13);
\draw[purpleline] (v13) -- (v1);

\draw[orangeline] (v1) -- (v14);
\draw[orangeline] (v14) -- (v4);
\draw[orangeline] (v4) -- (v24);
\draw[orangeline] (v24) -- (v2);

\draw[tealline] (v3) -- (v34);
\draw[tealline] (v34) -- (v4);

\fill[purpledot] (v1) circle (4pt);
\fill[orangedot] (v1) circle (2.5pt);
\node[purplevertex] at (v12) {};
\fill[purpledot] (v2) circle (4pt);
\fill[orangedot] (v2) circle (2.5pt);
\node[purplevertex] at (v23) {};
\fill[purpledot] (v3) circle (4pt);
\fill[tealdot] (v3) circle (2.5pt);
\node[purplevertex] at (v13) {};

\node[orangevertex] at (v14) {};
\fill[orangedot] (v4) circle (4pt);
\fill[tealdot] (v4) circle (2.5pt);
\node[orangevertex] at (v24) {};
\node[tealvertex] at (v34) {};

\end{tikzpicture}
\caption{A convex ear decomposition of $\N(\mathrm{U}_{3, 4}, \G_{\max})$, for the uniform matroid on of rank $3$ on $4$ elements, with the maximal building set.}
\label{fig:CED-U34}
\end{figure}
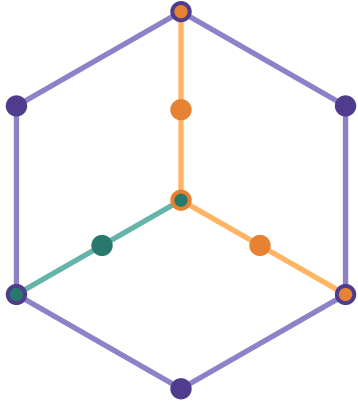

Let $\L$ be a geometric lattice, and let $\vartriangleleft$ be a total order on the set $\At(\L) = \Irr(\L)$ of atoms of $\L$, with induced edge-labeling $\lambda$. Here we consider a slight adaptation of the notions of Section~\ref{subsec:admissible-lattices}, where our edge-labeling takes values not in $\Z_{>0}$ but in the totally ordered set $(\At(\L), \vartriangleleft)$. This is similar to choosing an injective map $\At(\L) \hookrightarrow \Z_{>0}$, but, having the labels take values directly in $\At(\L)$ will simplify notations. Recall that Lemma \ref{lem:reducible} showed that for any built lattice $(\L, \G)$, we have a notion of restriction of $(\L, \G)$ at any element of $\L.$ When $\L$ is geometric, we have a finer notion of restriction, which allows us to restrict to any subset of atoms of $\L$. For any $S \subset \At(\L)$, let us denote by $\L_{|S}$ the sub-join-lattice of $\L$ generated by $S$. More concretely, this is the set of elements of $\L$ which can be written as a join of elements of $S.$  It is a classical fact of matroid theory that $\L_{|S}$ is again a geometric lattice. For instance, if $S$ is a basis of $\L$, then $\L_{|S}$ is a Boolean lattice. Let us also denote $\G_{|S} \coloneqq \G \cap \L_{|S}.$ The proof of the following lemma is analogous to that of Lemma \ref{lem:reducible} so we omit it. 
\begin{lemma}
The subset $\G_{|S} \subset \L_{|S}$ is a building set of $\L_{|S}.$
\end{lemma}

Recall that a broken circuit of $\L$, relative to the order $\vartriangleleft$, is a set of atoms of the form $C\setminus \min_{\vartriangleleft}C$, for any curcuit $C$ of $\L$. An \emph{nbc-basis} of $\L$ is a basis which does not contain any broken circuit. Denote by $B_1,\ldots, B_m$ the nbc-bases of $\L$ in lexicographic order with respect to $\vartriangleleft$. For $1 \leqslant i \leqslant m$, let us denote by $\Sigma_i$ the subcomplex $\N(\L_{|B_i}, \G_{|B_i}) \subset \N(\L,\G).$ More concretely, $\Sigma_i$ is the union of the simplices of $\N(\L, \G)$ whose vertices are elements of $\G$ which can be expressed as joins of elements in $B_i$. On the other hand, let us denote by $\Delta_i$ the union of the facets $\S$ of $\Sigma_i$ such that we have $\lambda(\S) = B_i$, where 
$$\lambda(\S) \coloneqq \{\lambda(G) \, | \, G \in \S\cup \un\}$$
(see Section~\ref{subsec:h-polynomial} for the definition of $\lambda(G)$).
The rest of this section will be devoted to proving the following theorem. 

\newtheorem*{thm:intro2}{Theorem~\ref{thm:cvx-ear-dec}}
\begin{thm:intro2}
    The sequence $\Delta_1, \ldots, \Delta_m$ is a convex ear decomposition of $\N(\L, \G).$
\end{thm:intro2}

This theorem is implied by the following series of lemmas, which follows the same line of argument as in \cite[Section 4]{nyman-swartz}.
\begin{lemma}
We have the equality $\N(\L, \G) = \bigcup_i\Delta_i.$    
\end{lemma}
\begin{proof}
It is enough to prove that for any facet $\S$ of $\N(\L, \G)$, the set of atoms $\lambda(\S)$ is an nbc-basis of $\L$. First notice that $\lambda(\S)$ is an $n$-element set whose join is $\un$, and so it is a basis. Assume by contradiction that $\lambda(\S)$ contains a broken circuit $C\setminus c$, with $c = \min_{\vartriangleleft} C$. By Lemma~\ref{lem:nested-tree}, the set $\S_{\geqslant \sigma(C)}$ has a unique minimum, where $\sigma(C)$ denotes the matroid closure of $C$, that is, $\sigma(C) = \bigvee C.$ We denote this minimum by $G.$ By the inequality $G\geqslant \sigma(C)$ and our assumption,  we have the inclusion $\lambda(\S_{\leqslant G}) \supset C\setminus c$. By minimality of $G$, $\sigma(C)$ is not below any element of $\max \S_{<G}$ and so it is not below $\bigvee \S_{<G}$, by the structural isomorphism $[\zero,\bigvee \S_{<G}]  \simeq \prod_{G' \in \max \S_{<G}}[\zero, G']$, since a circuit in a product of two matroids must belong to one of the two matroids. This implies that we must have $\lambda(G) \in C\setminus c$, but since $C$ is a circuit we have $G = c \vee \left(\bigvee \S_{<G}\right)$, which contradicts the inequality $c \vartriangleleft \min C\setminus c$ and the definition of $\lambda(G).$
\end{proof}
\begin{lemma}
We have $\Sigma_1 = \Delta_1$, and for  $2 \leqslant i \leqslant m$, we have $\Delta_i \subsetneq \Sigma_i$.
\end{lemma}
\begin{proof}
Notice that for every $i$, for every facet $\S$ of $\Sigma_i$ and all $G \in \S$ we have the inequality $\lambda(G) \trianglelefteqslant \lambda_{\L_{|B_i}}(G)$, where $\lambda_{\L_{|B_i}}$ means that we consider the nested set in the built lattice $(\L_{|B_i}, \G_{|B_i})$. This means that we have the inequality $\lambda(\S) \trianglelefteqslant \lambda_{\L_{|B_i}}(\S) = B_i$ in lexicographic order. As a consequence, when $i = 1$ we must have $\lambda(\S) = B_1$ and so we get the claimed equality $\Sigma_1 = \Delta_1.$ 

Let $1 \leqslant i \leqslant m$ be an index such that we have the equality $\Delta_i = \Sigma_i.$ We aim to prove that $i = 1$. Denote $B_1= \{b_1 \vartriangleleft \cdots \vartriangleleft b_n\}$, and suppose for contradiction that $k = \min \{ k \, | \, b_k \notin B_i\}.$ By definition of $B_1$, the elements of $B_i$ can be written as $b_1 \vartriangleleft \cdots \vartriangleleft b_{k-1} \vartriangleleft b'_k \vartriangleleft \cdots \vartriangleleft b'_n,$ for some atoms $b'_k, \cdots, b'_n,$ with $b_{k-1} \vartriangleleft b_k \vartriangleleft b'_{k}$. Consider the following facet of $\Sigma_i$
\[
\S = \{b_1\}\circ \cdots \circ \{b_1\vee \cdots \vee b_{k-1}\}\circ \{b_1\vee \cdots \vee b_{k-1}\vee b'_{k}\}\circ \cdots \circ \{b_1\vee \cdots \vee b_{k-1}\vee b'_k\vee \cdots \vee b'_{n-1}\}.
\]  Since $\Delta_i= \Sigma_i$ we have $\lambda(\S) = B_i.$ By Lemma \ref{lem:nested-tree}, the set $\S_{\geqslant b_k}$ has a unique minimum which we denote by $G$. By minimality of $G$, the atom $b_k$ is not below any element of $\max \S_{<G}$. Additionally, by the structural isomorphism of $\G$ at $\bigvee \S_{< G}$,
\[
[\zero, \bigvee \S_{<G}]\simeq \prod_{G' \in \max \S_{<G}}[\zero,G'], 
\] the atom $b_k$ is not below $\bigvee \S_{<G}$. Therefore, we have the equality $\left(\bigvee \S_{<G}\right) \vee b_k = G.$ By definition, this implies that we have $\lambda(G)  \vartriangleleft b_k$, and so we must have $\lambda(G) = b_j$ for some $1 \leqslant j \leqslant k-1$. Because of the order of composition that we have chosen for $\S$, this implies that we have $\lambda(\S_{\leqslant G}) \subset \{b_1,\ldots, b_{k-1}\},$ but then this implies $b_k \leqslant b_1\vee \cdots \vee b_{k-1}$ which is a contradiction. 
\end{proof}
\begin{lemma}\label{lem:ball-sphere}
The simplicial complex $\Delta_1$ is isomorphic to the boundary complex of an $(n-2)$-simplicial polytope, and for $2 \leqslant i \leqslant m$ the simplicial complex $\Delta_i$ is an $(n-2)$-ball which is isomorphic to a proper subcomplex of the boundary complex of an $(n-2)$-simplicial polytope.     
\end{lemma}

Before starting the proof of that lemma, we introduce a new notation. For any maximal nested set $\S\in \N_{\max}(\L,\G)$, recall from Lemma \ref{lem:nested-tree} that $\S$ has the structure of a tree, and that each node $G\in \S\cup \un$ of that tree has a label $\lambda(G)$. From this we can define the word $\w(\S)$ in the alphabet $\lambda(\S)$ by first reading the maximal letter $\lambda(G)$ among all leaves of $\S$, and then read the other letters inductively. More precisely, we set by induction
$$ \w(\L) \coloneqq \lambda(G_0)\w(\S\setminus G_0),$$
where $G_0 \in \min \S$ is such that $\lambda(G_0) = \max_{\vartriangleleft} \{\lambda(G) \, | \, G \in \min \S\}.$ For instance, if $\S$ is a chain of elements, then $\w(\S)$ simply reads the letters $\lambda(\S)$ from bottom to top.  
\begin{proof}[Proof of Lemma \ref{lem:ball-sphere}.]
If $i=1$, by the previous lemma we have $\Delta_1 \simeq \N(\L_{|B_1}, \G_{|\B_1})$. By \cite[Theorem~4.2]{feichtner-muller}, the simplicial complex $\N(\L_{|B_1}, \G_{|B_1})$ can be obtained as a sequence of stellar subdivisions of $\N(\L_{|B_1}, \un \cup B_1)$, which is the boundary of the $(n-1)$-simplex. 

For $2 \leqslant i \leqslant m$, by the previous lemma we have $\Delta_i \subsetneq \Sigma_i,$ where $\Sigma_i$ is the boundary complex of an $(n-2)$-simplicial polytope, for the same reasons as explained above. Thus, it is enough to show that $\Delta_i$ is nonempty and shellable. Denote by $b_1 \vartriangleleft \cdots \vartriangleleft b_n$ the elements of $B_i$, and denote $\S = \{b_n\}\circ \cdots \circ \{b_n\vee \cdots \vee b_2\}$. Let us prove that we have $\lambda(\S) = B_i,$ which will imply that $\Delta_i$ is nonempty. By contradiction denote by $G$ an element of $\S$ such that $\lambda(G)$ is not in $B_i$. By construction of $\S$ we have $\bigvee \S_{<G}\vee b_i = G$ for some $i$, with $\bigvee \S_{<G} = \bigvee_{j\in  J} b_j$ where $J$ is some set of indices all strictly larger than $i.$ If $\lambda(G) = \lambda(\bigvee \S_{<G}, G)$ is not $b_i$ then it must be some atom $b \vartriangleleft b_i.$ The fundamental circuit of $b$ with respect to $\{b_i\} \cup \{b_j \, | \, j \in J\}$ gives a broken circuit contained in $B_j$ which contradicts the fact that $B_j$ is an nbc-basis. 

We now turn our attention toward the shellability of $\Delta_i.$ Let us denote by $\lexop$ the total order on the facets of $\Delta_i$ given by the reverse lexicographic order on the words $\w(\S)$. Let us prove that $\lexop$ is a shelling order of $\Delta_i.$ Let $\S \lexop \S'$ be two facets of $\Delta_i.$ We need to find a third facet $\S''$ of $\Delta_i$ such that we have the inclusion $\S\cap \S' \subset \S''$, the inequality $\S'' \propto \S',$ and $\S''\cap \S'$ has cardinality $|\S'| - 1.$ Let us consider the first place at which the words $\w(\S)$ and $\w(\S')$ differ, and denote by $\lambda(G)$, $\lambda(G')$ the letters of those words at that place respectively. By assumption we have $\lambda(G') \vartriangleleft \lambda(G).$ Since $\w(\S)$ and $\w(\S')$ agree on the letters before, we have the equality $\S_{<G} = \S'_{<G'}.$ Let us denote by $G''$ the unique minimum of $(\S\cap \S')_{>G\vee G'}.$ We have the equality $\bigvee (\S\cap \S')_{<G} = \bigvee \S_{<G} = \bigvee \S'_{<G'}.$ We denote by $F$ this element. By Theorem \ref{thm:composition-nested-sets}, there exist nested sets $(\S_G)_G$ and $(\S'_G)_G$ in each local interval of $\S\cap \S'$ such that $\S = (\S\cap\S')\circ (\S_G)_G$ and $(\S\cap\S')\circ(\S'_G)_G$. In the local interval $([F, G''], \G^F_{G''})$, we have the inequality $\S_{G''} \lexop \S'_{G''}$. This implies the following claim. 

\medskip 

\noindent \textbf{Claim.} The nested set $\S'_{G''}$ has an ascent.\\
\emph{Proof.} By restriction and contraction we can assume $G'' = \un $ and $F = \zero.$ First notice that by Lemma \ref{lem:composition-descent}, the word associated to the facet $\S_0 = \{b_n\}\circ \cdots \circ \{b_n\vee \cdots \vee b_2\} \in \Delta_i$ is $\w(\S_0) = b_nb_{n-1}\ldots b_1,$ which is minimal for $\lexop.$ Furthermore, notice that $\S_0$ has only descents. By Proposition~\ref{prop:h-poly-descents} applied to the admissible map given by reversing $\vartriangleleft$, we have that the number of maximal nested sets of $\N(\L_{|B_i}, \G_{|B_i})$ with only descents is $1$, and so $\S_0$ is the unique maximal nested set with only descents. By the inequality $\S_{G''} \lexop \S'_{G''}$ we see that $\S'_{G''} \neq \S_0$ and so $\S'_{G''}$ has an ascent, which concludes the proof of the claim. 

\medskip

Let us denote by $A$ some ascent of $\S'_{G''}$. Denote by $P$ the parent of $A$ in $\S'_{G''}$, and denote $b_A = \lambda(A, \S'_{G''}), b_P = \lambda(P, \S'_{G''}),$ which are atoms in the basis $B_i$. The ascent assumption is exactly $b_A \vartriangleleft b_P.$ The desired facet $\S''$ will be obtained by swapping $b_P$ and $b_A$. 

Let us denote by $C$ the element $\bigvee (\S'_{G''}\setminus A)_{<P}$ and $\Tilde{\S'}_{G''} = (\S'_{G''}\setminus A)\circ(\{C\vee b_P\})$. In addition, we set $\Tilde{\S'}_G \coloneqq \S'_G$ for any $G \in \S \cap \S'$ different from $G''$. Finally, we set $\S'' \coloneq (\S\cap \S')\circ (\Tilde{\S'}_G)_G$. By construction we have $\S''\cap \S = \S' \cap \S, \S''\lexop \S'$ and $|\S'\cap \S''| = |\S'| - 1$. What remains is proving that the facet $\S''$ belongs to the subcomplex $\Delta_i$. This is equivalent to proving the equalities $\lambda(C, C\vee b_P) = b_P$ and $\lambda(C\vee b_P,P) = b_A$, knowing that we have $b_A = \lambda(C, C\vee b_A)$ and $b_P = \lambda(C\vee b_A, P) = b_P$ (by Lemma \ref{lem:composition-descent}). This is a simplified version of \cite[Lemma 4.3]{nyman-swartz}. For completeness we recall the argument here. 

Denote $b'_P = \lambda(C, C\vee b_P)$ and $b'_A = \lambda(C\vee b_P, P)$. By definition we have $b'_P \trianglelefteqslant b_P$ and $b'_A \trianglelefteqslant b_A$. Since $b'_P$ cannot be below $C\vee b_A$ (otherwise we would have $C\vee b_A = P$), we have $P = C\vee b_A \vee b'_P$, and so $b'_P\trianglelefteqslant b_P$ which implies $b_P = b'_P$. For $b'_A$, if $C\vee b'_A = C\vee b_A$ then we have $b_A \trianglelefteqslant b'_A$ by definition of $b_A$. If $C\vee b_A \neq C\vee b'_A$ we have $C\vee b_A \vee b'_A = P$ and so we get $b_P \trianglelefteqslant b'_A$ which contradicts $b_A \vartriangleleft b_P$.
\end{proof}
\begin{lemma}
For $2 \leqslant i \leqslant m$, we have the equality of simplicial complexes 
\[
\Delta_i \cap \left(\bigcup_{1 \leqslant j < i}\Delta_j\right) = \partial\Delta_i.
\] 
\end{lemma}
\begin{proof}
The boundary $\partial \Delta_i$ is the set of simplices of $\Delta_i$ which are included in a facet of $\Sigma_i$ not in $\Delta_i$. For a facet $\S$ of $\Sigma_i$ not in $\Delta_i$ we have $\lambda(\S) < B_i$, which gives the inclusion $\partial\Delta_i \subset \Delta_i \cap (\bigcup_{1 \leqslant j < i}\Delta_j).$ 

For any nested set $\S$ in $\Delta_i \cap (\bigcup_{1 \leqslant j < i}\Delta_j)$, by definition there exists an nbc-basis $B_j$ with $j<i$ (meaning that $B_j$ is before $B_i$ lexicographically) such that $\S \subset \N(\L_{|B_j}, \G_{|B_j})$. There exists an element $G_0$ in $\S$ such that $B_j \cap (G_0 \setminus \bigvee \S_{<G_0})$ is lexicographically strictly smaller than $B_i \cap (G_0 \setminus \bigvee \S_{<G_0})$ (where we view $G_0$ and $\bigvee \S_{<G_0}$ as the set of atoms below them). Let us denote $B_i \cap (G_0 \setminus \bigvee \S_{<G_0}) = \{b_1\vartriangleleft \cdots\vartriangleleft b_k\}$. The sequence of labels of the saturated chain 
$$ \bigvee \S_{<G_0} \lessdot \bigvee \S_{<G_0}\vee b_1 \lessdot \cdots \lessdot G_0 = \bigvee \S_{<G_0}\vee b_1\vee \cdots \vee b_k $$
is not lexicographically minimal, and so it admits a descent since $\lambda_{\omega}$ is an EL-labeling. This proves that the set of labels of that chain is not $\{b_1, \ldots, b_k\}$, which proves that if we denote $$\S_{G_0} = \{\bigvee \S_{<G_0}\vee b_1\}\circ \cdots \circ \{\bigvee \S_{<G_0}\vee b_1 \vee \cdots \vee b_{k-1}\},$$ we have $\lambda(\S_{G_0}) \neq \{b_1, \ldots, b_k\}.$ If we let $\S_G$ be any facet of the local interval $L^{G}(\S)$ in $(\L_{|B_i}, \G_{|B_i})$ for any $G \in \S\setminus G_0$, we see that the facet $\S\circ(\S_G)_G$ is a nested set in $(\L_{|B_i}, \G_{|B_i})$ containing $\S$, and such that $\lambda(\S\circ (\S_G)_G) \neq B_i,$ which shows that $\S \subset \partial \Delta_i$. 
\end{proof}

\section{Braid matroids and the minimal building set}\label{sec:braid-with-min}

In this section, we focus our attention on the face enumeration of the nested set complex $\mathcal{N}(\Pi_n,\mathcal{G}_{\min})$ arising from the lattice $\Pi_n$ with its minimal building set. Recall that $\Pi_n$ is the lattice of flats of the $n$-th \emph{braid matroid}, i.e., the matroid associated to the braid arrangement of type $A_n$. Equivalently, $\Pi_n$ is the lattice of flats of the complete graph on $n$ vertices, also identified with the partition lattice on $[n]$ with the partial order given by coarsening. The minimal building set $\G_{\min}$ consists of all nonempty connected flats, which correspond to partitions of $[n]$ into nonempty parts with exactly one nonsingleton part. 

\begin{figure}[ht]
\begin{tikzpicture}[scale=0.8]
    \begin{scope}[shift={(-4,0)}]
        
        \foreach \i in {0,1,2,3,4} {
            \coordinate (O\i) at ({90 + \i*72}:2.2);
        }
        
        \foreach \i in {0,1,2,3,4} {
            \coordinate (I\i) at ({90 + \i*72}:1.0);
        }
        
        \draw[thick] (O0) -- (O1) -- (O2) -- (O3) -- (O4) -- (O0);
        
        \draw[thick] (I0) -- (I2) -- (I4) -- (I1) -- (I3) -- (I0);
        
        \foreach \i in {0,1,2,3,4} {
            \draw[thick] (O\i) -- (I\i);
        }
        
        \foreach \i in {0,1,2,3,4} {
            \fill[black] (O\i) circle (3pt);
            \fill[black] (I\i) circle (3pt);
        }
        
        \node[font=\small] at (0, -3) {10 vertices, 15 edges};
    \end{scope}
    
    \begin{scope}[shift={(4.5,0)}]
        
        \foreach \i in {0,1,2,3,4} {
            \coordinate (O\i) at ({90 + \i*72}:2.2);
        }
        
        \foreach \i in {0,1,2,3,4} {
            \coordinate (I\i) at ({90 + \i*72}:1.0);
        }
        
        \coordinate (C) at (2.8, 2.0);
        
        \fill[blue!15] (C) -- (O2) -- (O3) -- cycle;
        \fill[blue!12] (C) -- (O3) -- (O4) -- cycle;
        \fill[blue!10] (C) -- (O1) -- (O2) -- cycle;
        \fill[blue!18] (C) -- (O4) -- (O0) -- cycle;
        \fill[blue!20] (C) -- (O0) -- (O1) -- cycle;
        
        \fill[blue!22] (C) -- (I0) -- (I2) -- cycle;
        \fill[blue!18] (C) -- (I2) -- (I4) -- cycle;
        \fill[blue!14] (C) -- (I4) -- (I1) -- cycle;
        \fill[blue!16] (C) -- (I1) -- (I3) -- cycle;
        \fill[blue!20] (C) -- (I3) -- (I0) -- cycle;
        
        \fill[blue!25] (C) -- (O0) -- (I0) -- cycle;
        \fill[blue!22] (C) -- (O1) -- (I1) -- cycle;
        \fill[blue!12] (C) -- (O2) -- (I2) -- cycle;
        \fill[blue!10] (C) -- (O3) -- (I3) -- cycle;
        \fill[blue!15] (C) -- (O4) -- (I4) -- cycle;
        
        \draw[thick] (O0) -- (O1) -- (O2) -- (O3) -- (O4) -- (O0);
        
        \draw[thick] (I0) -- (I2) -- (I4) -- (I1) -- (I3) -- (I0);
        
        \foreach \i in {0,1,2,3,4} {
            \draw[thick] (O\i) -- (I\i);
        }
        
        \foreach \i in {0,1,2,3,4} {
            \draw[thick, gray] (C) -- (O\i);
            \draw[thick, gray] (C) -- (I\i);
        }
        
        \foreach \i in {0,1,2,3,4} {
            \fill[black] (O\i) circle (3pt);
            \fill[black] (I\i) circle (3pt);
        }
        
        \fill[pink] (C) circle (3pt);
        \draw[thick, magenta!70!black] (C) circle (3pt);
        
        \node[font=\small] at (0, -3) {11 vertices, 25 edges, 15 triangles};
    \end{scope}
\end{tikzpicture}
\caption{The nested set complex $\calN(\Pi_4, \G_{\min})$: Petersen graph, and the cone thereof. By a classical result of Vogtmann, the simplicial complex $\calN(\Pi_4, \G_{\min})$ is homotopic to a wedge of $6$ spheres of dimension~$1$.}
\label{fig:petersen}
\end{figure}


This complex is the main subject of study in \cite{feichtner-braid}. In that paper, Feichtner showed that this complex is (after deconing) isomorphic to the so-called \emph{complex of trees} $T_n$ (see \cite[Section~3]{feichtner-braid} for the relevant definitions). The complexes $T_n$ (and therefore the complexes $\N(\Pi_n,\mathcal{G}_n)$) were proved to be shellable by Trappmann and Ziegler in \cite{trappmann-ziegler} (another independent unpublished proof is attributed to Wachs in \cite{feichtner-braid}). As a special case of Theorem~\ref{thm:vertex decomposability} we deduce the following extension to Trappmann--Ziegler's and Wach's result.

\begin{corollary}\label{coro:trees-vd}
    For every $n\geq 1$ the complex of trees $T_n$ is vertex decomposable and admits a convex ear decomposition.
\end{corollary}

Furthermore, thanks to Theorem~\ref{thm:cvx-ear-dec}, we can also deduce inequalities among the $h$-vector entries of $T_n$, many of which, to the best of our knowledge, are new. Before stating these inequalities, we formulate explicitly a result that appears implicitly in the literature (cf. the work by Elvey Price and Sokal \cite{price-sokal}). 

\begin{thm}\label{thm:second-eulerian}
    The $h$-polynomial of the nested set complex $\mathcal{N}(\Pi_{n+1},\mathcal{G}_{\min})$ equals the $n$-th second Eulerian polynomial.
\end{thm}

For the sake of providing a self-contained and explicit treatment of this fact, we include below a new proof based on the formula given by Proposition~\ref{prop:h-poly-descents}. Second Eulerian polynomials were introduced by Gessel and Stanley \cite{gessel-stanley} in their study of Stirling permutations. Let us recapitulate the essentials about these objects. For a fixed positive integer $n\geq 1$, let us consider the multiset $X_n = \{1,1,2,2,\ldots,n,n\}$ (every integer from $1$ to $n$ appears exactly twice). A permutation $\pi$ of the elements of $X_n$ is said to be a \emph{Stirling permutation} if, for every $i\in \{1,\ldots,n\}$, the numbers that appear between the two ocurrences of $i$ in $\pi$ (when written in one-line notation) are all larger than $i$. For example, for $n=2$ there are exactly three Stirling permutations, written in one-line notation as $1122$, $1221$, and $2211$. Following B\'ona's notation \cite{bona}, we will write $Q_n$ to denote the set of all Stirling permutations on $\{1,1,\ldots,n,n\}$. 

A \emph{descent} on a Stirling permutation $\pi \in Q_n$ is an index $1\leqslant i\leqslant 2n-1$ such that $\pi_i > \pi_{i+1}$ (where $\pi_i$ stands for the $i$-th entry of $\pi$ in one-line notation). We denote by $\des(\pi)$ the number of descents of $\pi$. In analogy with the classical Eulerian polynomials, the $n$-th \emph{Second Eulerian polynomial} is defined by
    \[ P_n(x) = \sum_{\pi\in Q_n} x^{\des(\pi)}.\]
The first few values are:
    \[ P_n(x) = \begin{cases}
        1 & n = 1,\\
        2x + 1 & n = 2,\\
        6x^2+8x+1 & n = 3,\\
        24x^3 + 58x^2+22x+1 & n = 4,\\
        120x^4 + 444x^3 + 328x^2 + 52x + 1 & n = 5.
    \end{cases}
    \]

\begin{example}\label{ex:braid-4-min}
Let $\Delta$ be the nested set complex $\calN(\Pi_4, \G_{\min})$ which is the $1$-dimensional simplicial complex of the Petersen graph, and $\mathrm{cone}(\Delta)$ its cone of pure dimension $2$; see \Cref{fig:petersen}. By counting faces, the $f$-vector of $\Delta$ is $f(\Delta) = (1, 10, 15)$. Using
\[
\sum_{i=0}^{d} f_i(\Delta) (t-1)^{d-i} = \sum_{i=0}^{d} h_i(\Delta) t^{d-i},
\]
with $d = \dim \Delta + 1 = 2$, we compute
\begin{align*}
&f_0(t-1)^2 + f_1(t-1) + f_2 \\
&= 1 \cdot (t^2 - 2t + 1) + 10 \cdot (t - 1) + 15 \\
&= t^2 + 8t + 6,
\end{align*}
giving $h(\Delta) = (1, 8, 6)$ and $h$-polynomial $h(\Delta; t) = 6t^2 + 8t + 1$. Similarly, 
\[
f(\mathrm{cone}(\Delta)) = (1, 11, 25, 15), \quad \text{ and } \quad h(\mathrm{cone}(\Delta)) = (1, 8, 6, 0).
\]
\end{example}

\medskip

\begin{proof}[Proof of Theorem~\ref{thm:second-eulerian}]
By Proposition \ref{prop:h-poly-descents}, it is enough to find an injective admissible map on $\Pi_{n+1}$, and a bijection between the set $\N_{\max}(\Pi_{n+1},\G_{\min})$ of maximal nested sets of $(\Pi_{n+1}, \G_{\min})$ and the set $\StP_n$ of Stirling permutations on $[n]$, which preserves the number of descents on both sides. For inductive purposes, it will be convenient to prove a more functorial result. For any finite set $I$ of cardinality $\geqslant 2$ we denote by $\Pi_{I}$ the set of partitions of $I$ ordered by coarsening. For any finite subset $I\subset \Z$ of cardinality $\geqslant 2$, we denote by $\StP_{I}$ the set of Stirling permutations on the set $I$. For any finite subset $I\subset \Z$ of cardinality $\geqslant 2$, we will construct a bijection 
$$\Psi_I: \N_{\max}(\Pi_{I}, \G_{\min}) \xrightarrow{\sim} \StP_{I\setminus \min I}, $$
by induction on the cardinality of $I$. 

The elements of $\G_{\min}\subset \Pi_{I}$ are the partitions of $I$ with exactly one nonsingleton equivalence class. We will identify those elements with this equivalence class, which is any subset of $I$ of cardinality $\geqslant 2$. For any $\S \in \N_{\max}(\Pi_{I},\G_{\min})$, let us consider the set $\S_0 = \{ G \in \S \, | \, \min I \in G\}$. Note that $\bigvee \min \S_0$ belongs to $\G_{\min}$ (its only nonsingleton equivalence class is $\bigcup_{G\in \min \S_0}G$), and so since $\S$ is nested, $\S_0$ must have a unique minimum, which will be denoted $G_0$. 

\medskip

\textbf{Claim.} Either $\S_{<G_0}$ is empty, in which case $G_0$ is an atom, either $\S_{<G_0}$ has a unique maximum, which is $G_0\setminus \min I$. 
\smallskip

\noindent \emph{Proof of Claim.} Since $\S$ is maximal, by Lemma \ref{lem:max-nested-purity} the local interval $L^{G_0}(\S)$ has rank $1$ and so either $G_0$ is an atom, either $\max \S_{<G_0}$ has one element, which must be $G_0\setminus \min I$ by minimality of $G_0$, either $\max \S_{<G_0}$ has two elements say $G_1, G_2$, with $G_1 \cap G_2 = \emptyset$ and $G_1 \cup G_2 =G$. However, in the last case we would have $\min I\in G_1$ or $\min I\in G_2$ which would contradict the minimality of $G_0$. 

\medskip

We denote $\Tilde{G_0} = G_0 \setminus \min I.$ We define a map $\Psi_I: \N_{\max}(\Pi_{I},\G_{\min}) \rightarrow \StP_{I\setminus \min I}$ by the following inductive formula, in one-line notation:
\begin{equation}\label{eq:def-bijection}
    \Psi_I(\S) = (\min \Tilde{G_0})\, \Psi_{\Tilde{G_0}}(\S_{< \Tilde{G_0}})\, (\min \Tilde{G_0})\, \Psi_{*\cup I\setminus G_0}((G_0 \vee \S)\setminus G_0).
\end{equation}
If $G_0$ is an atom, then $\Psi_{\Tilde{G_0}}(\S_{< \Tilde{G_0}})$ means the empty Stirling permutation by convention. Otherwise, the set $\S_{< \Tilde{G_0}}$ is a maximal nested set of $([\zero, \Tilde{G_0}], \G_{\min}),$ which is isomorphic to $(\Pi_{\Tilde{G_0}},\G_{\min}),$ so $\Psi_{\Tilde{G_0}}(\S_{< \Tilde{G_0}})$ makes sense by induction. 

If $G_0 = \un$ then $\Psi_{*\cup I\setminus G_0}((G_0 \vee \S)\setminus G_0)$ means the empty Stirling permutation by convention. Otherwise, by Theorem \ref{thm:composition-nested-sets} (iv), the set $(G_0 \vee \S)\setminus G_0$ is a nested set of $([G_0, \un], (\G_{\min})_{G_0})$ which is isomorphic to $(\Pi_{*\cup I\setminus G_0}, \G_{\min})$ (where we set $*$ to be any integer strictly less than $\min I$), so $\Psi_{*\cup I\setminus G_0}((G_0 \vee \S)\setminus G_0)$ makes sense by induction. 

Since by induction $\Psi_{\Tilde{G_0}}(\S_{< \Tilde{G_0}})$ is a Stirling permutation of elements strictly above $\min \Tilde{G_0}$, and $\Psi_{*\cup I\setminus G_0}((G_0 \vee \S)\setminus G_0)$ is a Stirling permutation of elements distinct from the elements of $\Tilde{G_0}$, the formula \eqref{eq:def-bijection} defines a Stirling permutation. 

To prove that $\Psi_{I}$ is a bijection for any $I$, we will construct an inverse $\Phi_I: \StP_{I\setminus \min I} \rightarrow \N_{\max}(\Pi_{I}, \G_{\min}),$ also by induction. Any Stirling permutation $\sigma$ of $I\setminus \min I$ can be uniquely written in one-line notation as $\sigma = \alpha \sigma_1 \alpha \sigma_2$ for some $\alpha \in I \setminus \min I$, and some Stirling permutations $\sigma_1$, $\sigma_2$, on subsets $I_1, I_2$ such that $I_1 \sqcup I_2 \sqcup \alpha = I\setminus \min I.$ We then set 
\begin{equation}\label{eq:def-inverse}
    \Phi_{I}(\sigma) = \{I_1\cup \alpha \}\circ(\Phi_{I_1}(\sigma_1), \Phi_{\ast \cup I_2}(\sigma_2))
\end{equation}
(see Theorem~\ref{thm:composition-nested-sets} for the definition of the nested set composition). If $I_1$ or $I_2$ is empty, we adopt the necessary conventions. Otherwise, the nested set $\Phi_{\ast \cup I_2}(\sigma_2)$ lives in $(\Pi_{\ast \cup I_2}, \G_{\min})$ which can be identified with $([I_1\cup \alpha \cup \min I, \un], \G_{\min}),$ so the composition of nested sets makes sense. Theorem \ref{thm:composition-nested-sets} (iv) implies that $\Psi_I$ and $\Phi_I$ are inverses of each other, again by induction. 

What is left is to find an injective admissible map on $\Pi_I$ such that $\Psi_I$ preserves the descent number on both sides, for any finite subset $I \subset \Z$ of cardinality $\geqslant 2$. The join-irreducible elements of $\Pi_{I}$ are the atoms of $\Pi_{I}$, which are the partitions $|ij|$ which have only one nonsingleton equivalence class $\{i,j\}$, where $i,j \in I$. Let $\omega:\At(\Pi_{I}) \hookrightarrow \Z_{>0}$ be any injective admissible map such that $\omega(|i < j|) < \omega(|i' < j'|)$ if $i< i'$ or if $i= i'$ and $j< j'$. Since $\Pi_I$ is semimodular, there exists such an admissible map (see \cite{stanley-74}). For any $\S\in \N_{\max}(\Pi_{I}, \G_{\min})$, by definition of the descent number of a Stirling permutation, if $\Psi_{\Tilde{G_0}}(\S_{< \Tilde{G_0}})$ is nonempty we have 
$$ \des(\Psi_I(\S)) = \des(\Psi_{\Tilde{G_0}}(\S_{< \Tilde{G_0}})) + \des(\Psi_{*\cup I\setminus G_0}((G_0 \vee \S)\setminus G_0)) + \varepsilon_{G_0} + 1 $$
where $\varepsilon_{G_0} = 1$ if $\min \Tilde{G_0}$ is strictly greater than the first value of $\Psi_{*\cup I\setminus G_0}((G_0 \vee \S)\setminus G_0)$ and $\varepsilon_{G_0} = 0$ otherwise. If $\Psi_{\Tilde{G_0}}(\S_{< \Tilde{G_0}})$ is empty (i.e. $G_0$ is an atom) we have 
\[ \des(\Psi_I(\S)) = \des(\Psi_{\Tilde{G_0}}(\S_{< \Tilde{G_0}})) + \des(\Psi_{*\cup I\setminus G_0}((G_0 \vee \S)\setminus G_0)) + \varepsilon_{G_0}.\]

Let us now compute the descent number of $\S$ with respect to $\omega$ defined above. If $\Psi_{\Tilde{G_0}}(\S_{< \Tilde{G_0}})$ is nonempty (i.e. $G_0$ is not an atom), then $\Tilde{G_0}$ is a descent of $\S$ because we have $\lambda_{\omega}(G_0) = \omega(|\min I \, \min \Tilde{G_0}|)$ which is strictly smaller than $\lambda_{\omega}(\Tilde{G_0})$, since $\Tilde{G_0}$ is a set of elements all strictly greater than $\min I.$ Thus, in that case we have 
$$\des(\S) = \des(\S_{<\Tilde{G_0}}) + 1 + \des(\S_{\nleq G_0}) + \varepsilon'_{G_0},$$
where $\varepsilon'_{G_0}$ is $1$ if $G_0$ is a descent, and $0$ otherwise. By construction the label of the father of $G_0$ is $\omega(|\min I \, \beta|),$ where $\beta$ is the first value of $\Psi_{*\cup I\setminus G_0}((G_0 \vee \S)\setminus G_0)$. This means that we have $\varepsilon_{G_0} = \varepsilon'_{G_0}$. Moreover, by Lemma \ref{lem:composition-descent} we have $\des(\S_{\nleq G_0}) = \des((G_0 \vee \S) \setminus G_0)$ and so by induction we can conclude that $\des(\Psi_I(\S))$ and $\des(\S)$ are equal. By similar arguments we arrive at the same result when $\Psi_{\Tilde{G_0}}(\S_{< \Tilde{G_0}})$ is empty (i.e. when $G_0$ is an atom), which concludes the proof. 
\end{proof}

By putting together the results in Theorem~\ref{thm:cvx-ear-dec}, specifically the inequalities in \eqref{eq:top-heavy}, and Theorem~\ref{thm:second-eulerian}, we obtain the following corollary.

\begin{corollary}\label{coro:stirling-top-heavy}
    For every $0\leqslant k \leqslant n$, let us denote by $Q_{n,k}$ the set of Stirling permutations $\pi\in Q_n$ that have exactly $k$ descents. Then, $|Q_{n,k}| \leqslant |Q_{n,n-k}|$ for each $k\leqslant \lfloor n/2\rfloor$ and $|Q_{n,0}| \leqslant |Q_{n,1}| \leqslant \cdots \leqslant |Q_{n,\lfloor n/2\rfloor}|$.
\end{corollary}

It is natural to ask if the polynomials $P_{n}(x)$ are unimodal for each $n$. The answer is affirmative and, in fact, much more is true. Thanks to a result of B\'ona \cite{bona}, these polynomials are real-rooted. Real-rootedness implies (ultra) log-concavity without internal zeros, which in turn implies the unimodality of these polynomials (see~\cite{stanley-unimodality,branden2015unimodality} for more on these classes of inequalities). One of the consequences of the above corollary is that the mode is located at the ``top half'' of the polynomials. 

\section{Final remarks and open problems}

In this section we collect a few remarks and open problems.

\subsection{Unimodality, real-rootedness, and Chow polynomials}\label{sec:unimodality-real-roots}

As mentioned above, the nested complex $\N(\Pi_n,\mathcal{G}_{\min})$ has a real-rooted $h$-polynomial by work of B\'ona \cite{bona}. On the other hand, Athanasiadis and Kalampogia-Evangelinou conjectured \cite[Conjecture~1.1]{athanasiadis-kalampogia} that for every geometric lattice $\L$ the $h$-polynomial of $\N(\L,\G_{\max})$ is real-rooted; together with Douvropoulos, they further note \cite[Question~1.4]{athanasiadis-douvropoulos-kalampogia} that the problem already appears subtle for doubly Cohen--Macaulay lattices.

Real-rootedness fails if one does not impose restrictions on the building sets: combining Example~\ref{ex:augmented-bergman} with a construction of Athanasiadis and Ferroni \cite[Example~5.1]{athanasiadis-ferroni}, one finds geometric lattices $\L$ and building sets $\G$ for which the $h$-polynomial of $\N(\L,\G)$ is not even unimodal.

It is therefore natural to seek a conjecture that interpolates between B\'ona's real-rootedness result for second Eulerian polynomials \cite{bona} and the real-rootedness conjecture by Athanasiadis and Kalampogia-Evangelinou \cite{athanasiadis-kalampogia}. We propose the following.

\begin{conjecture}\label{conj:flag-real-rooted}
    Let $\L$ be a geometric lattice and let $\G$ be a building set. If the nested set complex $\N(\L,\G)$ is flag, then its $h$-polynomial is real-rooted.
\end{conjecture}

Recall that a simplicial complex is \emph{flag} if all its minimal nonfaces have size at most $2$ (equivalently, if it is the clique complex of a graph). From the definitions, $\N(\L,\G_{\max})$ is flag for every $\L$, since it is an order complex. One also checks without difficulty that $\N(\Pi_n,\G_{\min})$ is flag. The connection between flagness and real-rootedness phenomena in combinatorics remains somewhat mysterious; see \cite[Section~4]{reiner-welker}. Although Conjecture~\ref{conj:flag-real-rooted} may well admit counterexamples (in particular given the scarcity of computational evidence and the difficulty of producing supporting data), we believe that flagness of $\N(\L,\G)$ is an essential feature if one hopes to strengthen the inequalities of Corollary~\ref{cor:inequalities}.

We also briefly mention a related problem. To any pair $(\L,\G)$, with $\L$ a geometric lattice, one can associate the \emph{Chow polynomial} \cite{eur-ferroni-matherne-pagaria-vecchi}, defined as the Hilbert--Poincar\'e series of the Chow ring of $(\L,\G)$ \cite{feichtner-yuzvinsky}. For certain choices of $\L$ and $\G$ this polynomial is known to be real-rooted (see, e.g., \cite{coron-ferroni-li,branden-vecchi2} and the references therein). It would be very interesting to clarify how Chow polynomials relate to the $h$-polynomials of nested set complexes, especially in connection with Conjecture~\ref{conj:flag-real-rooted}. In \cite{coron-ferroni-li} we show that, for certain geometric lattices, the real-rootedness conjectures for the Chow polynomial of $(\L,\G_{\max})$ and for the $h$-polynomial of $\N(\L,\G_{\max})$ can be investigated in parallel.

\begin{question}
    Assume $\L$ is a geometric lattice. Is there a relationship between the $h$-polynomial of $\N(\L,\G)$ and the Chow polynomial arising from the pair $(\L,\G)$?
\end{question}

\subsection{Two problems on Stirling permutations}

Let us denote by $Q_{n,k}$ the set of all Stirling permutations on $2n$ elements having exactly $k$ descents. By Corollary~\ref{coro:stirling-top-heavy} we know that the following inequalities hold true:
    \[\text{$|Q_{n,k}| \leqslant |Q_{n,n-k}|$ for each $k\leqslant \lfloor n/2\rfloor$ and $|Q_{n,0}| \leqslant |Q_{n,1}| \leqslant \cdots \leqslant |Q_{n,\lfloor n/2\rfloor}|$.} \]

It would be very interesting to have a proof of the preceding inequalities via explicit injections $Q_{n,k}\hookrightarrow Q_{n,n-k}$ and, similarly, $Q_{n,0} \hookrightarrow Q_{n,1} \hookrightarrow \cdots \hookrightarrow Q_{n,\lfloor n/2\rfloor}$.

\begin{problem}
    Find suitable injections $Q_{n,k}\hookrightarrow Q_{n,n-k}$ for $k\leqslant \lfloor n/2\rfloor$ and $Q_{n,k}\hookrightarrow Q_{n,k+1}$ for $0\leqslant k \leqslant \lfloor n/2\rfloor - 1$.
\end{problem}

Now we turn our attention to another open question. A \emph{signed partition} $\pi$ is a partition of the set $\{0,\pm 1,\ldots, \pm n\}$ such that:
\begin{itemize}
    \item $P\in \pi$ if and only if $(-P) \in \pi$. 
    \item If $\{i,-i\} \subseteq P\in \pi$ for some $1\leqslant i\leqslant n$, then $0\in P$.
\end{itemize}

In the above list, the notation $(-P)$ stands for $\{-i : i \in P\}$. The second condition says that the only block of the partition allowed to contain a number and its opposite is the one that contains $0$. The poset of all signed partitions under coarsening is called the \emph{type B partition lattice} and, following \cite[Section~4.2]{athanasiadis-kalampogia}, we denote it by $\Pi_n^B$. As a poset, $\Pi_n^B$ coincides with the lattice of flats of the type $B$ braid arrangement, i.e., the hyperplane arrangement in $\mathbb{C}^n$ defined by the equations $x_i\pm x_j = 0$ for $1\leqslant i\leqslant j \leqslant n$.

\begin{question}
    Is there a combinatorial interpretation of the $h$-vector of $\N(\Pi_n^B,\mathcal{G}_{\min})$ that parallels Theorem~\ref{thm:second-eulerian}?
\end{question}

The $h$-polynomial of $\N(\Pi_n^B,\mathcal{G}_{\max})$ has been studied in \cite[Section~4.2]{athanasiadis-kalampogia} where it is in fact proved to be real-rooted. We emphasize the fact that $\N(\Pi_n^B,\G_{\min})$ is flag (since $\Pi_n^B$ is a supersolvable geometric lattice, this follows from \cite[Proposition~2.26]{Coron2023Supersolvability}), and thus in particular we expect that the $h$-polynomial of this complex is real-rooted as well.

\subsection{Convex ear decompositions, and shellability}

Given our results on nested sets complexes $\N(\L, \G)$ for a geometric lattice $\L$ with arbitrary building set $\G$ to be both vertex decomposable (and thus shellable) and admit convex ear decomposition, it is natural to ask the following question.

\begin{question}
    For which geometric lattices $\L$ and which building sets $\G$, does the nested set complex $\N(\L, \G)$ admit a PS ear decomposition? 
\end{question}

In general, $\N(\L, \G)$ does not admit a PS ear decomposition. For example, the order complex of $\mathcal{L}(U_{3,4})$ only admits the convex ear decomposition shown in \Cref{fig:CED-U34} up to isomorphism, which is not a PS ear decomposition. 

On the other hand, it is unclear whether general convex ear decompositions are enough to guarantee the shellability of a complex. In the literature there are several examples of complexes that possess convex ear decompositions (see, for example, Schweig's contributions \cite{schweig1,schweig2} and the work of Woodroofe \cite{woodroofe}), but all of them are known to be shellable. 

\begin{problem}
    Find a non shellable simplicial complex $\Delta$ admitting a convex ear decomposition.
\end{problem}

While we are inclined to expect such examples to abound, we have not been able to construct one.

\bibliographystyle{amsalpha}
\bibliography{biblio}

\end{document}